\documentclass{article}
\usepackage{amssymb}
\usepackage[margin=2.5cm]{geometry}
\usepackage{amsmath}
\usepackage{amsthm}
\usepackage{mathrsfs}
\usepackage{subcaption}
\usepackage{tikz}
\usepackage{listings}
\usepackage{float}
\usepackage{marvosym}

\usepackage{fancyhdr}
\newtheorem{defn}{Definition}[section]
\newtheorem{lemma}[defn]{Lemma}
\newtheorem{cor}[defn]{Corollary}
\newtheorem{prop}[defn]{Proposition}

\newtheorem{theorem}[defn]{Theorem}

\usetikzlibrary{positioning}
\usetikzlibrary{matrix,arrows}
\usetikzlibrary{shapes.geometric}
\usetikzlibrary{decorations.pathreplacing}
\usetikzlibrary{arrows.meta}

\makeatletter
\DeclareRobustCommand{\pder}[1]{%
  \@ifnextchar\bgroup{\@pder{#1}}{\@pder{}{#1}}}
\newcommand{\@pder}[2]{\frac{\partial#1}{\partial#2}}
\makeatother

\begin{document}

\title{Recurrence of the frog model on the 3,2-alternating tree}

\author{Josh Rosenberg}
\date{}
\maketitle

\begin{abstract}Consider a growing system of random walks on the 3,2-alternating tree, where generations of nodes alternate between having two and three children.  Any time a particle lands on a node which has not been visited previously, a new particle is activated at that node, and begins its own random walk.  The model described belongs to a class of problems that are collectively referred to as the frog model.  Building on a recent  proof of recurrence (meaning infinitely many frogs hit the root with probability one) on the regular binary tree, this paper establishes recurrence for the 3,2-alternating case.
\end{abstract}

\section{Introduction}

The frog model is a system involving a collection of branching random walks on a rooted graph.  One active frog is initially positioned at the root along with some distribution of sleeping frogs on the set of non-root vertices.  The active frog performs a discrete time nearest-neighbor random walk on the graph which activates whatever sleeping frogs reside on the vertices on which it lands.  Upon being activated, frogs perform their own independent discrete time nearest-neighbor random walks which also activate sleeping frogs in the same fashion.  

Perhaps the most fundamental question that can be asked about any version of the frog model is: Is it recurrent?  By this we mean: Do infinitely many frogs return to the root with probability one?  This question has been explored for a variety of different scenarios.  Among them, one of the most recent cases has been that of the frog model on the n-ary tree.  In \cite{HJJ1} Christopher Hoffman, Tobias Johnson, and Matthew Junge addressed the issue of recurrence for the frog model on the n-ary tree with one sleeping frog per non-root vertex.  They established that the model is transient (i.e. the probability infinitely many frogs return to the root is zero) on the n-ary tree for $n\geq 5$, and recurrent on the binary tree.  The cases of the 3-ary and 4-ary trees, which remain open, were conjectured in their paper to be recurrent and transient respectively.

Among the results achieved in \cite{HJJ1}, arguably the most significant was the proof of recurrence for the binary tree.  The proof involved first presenting an altered form of the frog model in which the frogs perform (sometimes) terminating, non-backtracking random walks.  This variant on the original model, which is shown to have nice self-similarity properties, is appropriately named the self-similar frog model.  After the introduction of the self-similar model it is shown how it can be coupled with the original model so that the number of returns to the root in the original case always dominates that of the self-similar, thus reducing the problem to establishing recurrence in the self-similar case.  From here, the self similarity of the altered model is exploited to show that the probability generating function for the number of returns to the root is a fixed point of an easily expressed operator.  The authors then use a technique which they call Poisson thinning to show that the only fixed point of this operator (among a broad class of functions defined on $[0, 1]$) is 0, thus demonstrating recurrence for both the self-similar and (see above) original models.

The difficulty in extending the above result to the 3-ary tree seems to derive from the 3-ary case (it appears) being very close to criticality.  Specifically, efforts to construct an alternate model that is dominated by the original and possesses a high degree of self similarity appear to produce cases that cease to be recurrent.  While we believe that such an approach can in theory still work for the 3-ary tree, the competing necessities of achieving some sort of self-similarity $\textit{and}$ preserving recurrence appear to make the application of an approach at all similar to the one employed for the binary tree computationally intractable.  In this paper we therefore address an intermediate case; that of the 3,2-alternating tree where each node belonging to an even numbered generation has three children (with the root defined as generation 0) and each node belonging to an odd numbered generation has two.  Though some of the same difficulties that arise in attempting to extend the result for $\mathbb{T}_2$ to $\mathbb{T}_3$ (the 2-ary and 3-ary trees respectively) still present themselves in the case of $\mathbb{T}_{3,2}$ (the 3,2-alternating tree), we are nevertheless able to adapt the approach used to establish recurrence on $\mathbb{T}_2$ to the demands of this somewhat more unwieldy variation.

\medskip
\noindent
{\bf Statement and discussion of main result.} In its examination of the frog model on $\mathbb{T}_{3,2}$, where we start with one sleeping frog per non-root vertex and where activated frogs perform unbiased random walks on the tree, this paper centers around establishing the following theorem.

\begin{theorem}\label{theorem:neat1} The frog model on $\mathbb{T}_{3,2}$ is recurrent. 
\end{theorem}

This result is achieved by first introducing the non-backtracking frog model on $\mathbb{T}_{3,2}$ in which individual frogs perform uniformly random non-backtracking walks (i.e. at each step a frog chooses randomly from the set of all adjacent vertices $\textit{except}$ the one from which it just came) that are stopped at the root.  Letting $Z$ and $V'$ represent the number of times the root is landed on in the original and non-backtracking models respectively, we then show how the two models can be coupled in such a way that the path of each individual frog in the non-backtracking model is a subset of the path taken by the corresponding frog (meaning the frog that originated in the same location) in the original model.  From this it then follows that $Z$ stochastically dominates $V'$.  From here a variant of the non-backtracking model is introduced in which certain carefully chosen restrictions are placed on the number of frogs which can move down an edge (i.e. away from the root) without being stopped.  This new model, to be referred to as the self-similar model on account of its possessing certain self-similarity properties which we later establish, is then coupled with the original model via a natural coupling with the non-backtracking case, thus establishing that $V$ (the number of times the root is hit in the self-similar case) is dominated by $Z$.

Having reduced the task of proving Theorem \ref{theorem:neat1} to establishing recurrence for the self similar model, the model's self similarity properties are then exploited in order to show that $f$ (the probability generating function for $V$) is a fixed point for a rather complicated operator $\mathcal{A}$.  After showing that $\mathcal{A}$ possesses an important monotonicity property when applied to a large class of functions $S$, we then show that $\mathcal{A}^n 1\rightarrow 0$ on $[0, 1)$ as $n\rightarrow\infty$.  This last task involves a substantial amount of computation as $\mathcal{A}$ is iteratively applied to the probability generating functions for a series of Poisson distributions with gradually increasing means.  Thus, computer assistance is required at a certain juncture.  However, once this is accomplished it follows from the monotonicity property of $\mathcal{A}$ that $\mathcal{A}^n f\rightarrow 0$ on $[0, 1)$ which, because $\mathcal{A}f=f$, implies that $f(x)=0$ on $[0, 1)$.  From this the recurrence of the self-similar model immediately follows, thus completing the proof of Theorem \ref{theorem:neat1}.

\section{Recurrence for $\mathbb{T}_{3,2}$}

2.1. {\bf The non-backtracking frog model.}  In constructing the non-backtracking frog model and establishing a coupling with the ordinary model, we first define the Markov process $\Upsilon:\mathbb{N}\rightarrow\mathbb{T}_{3,2}$ as follows: If $\Upsilon(0)=\varnothing$ (where $\varnothing$ represents the root) then $\Upsilon$ simply proceeds as an unbiased random walk on $\mathbb{T}_{3,2}$.  If $\Upsilon(0)\neq\varnothing$ then $\Upsilon$ proceeds as an unbiased random walk $\textit{except}$ that if $\Upsilon(n)=\varnothing$ (for some n) and $\Upsilon(n+1)$ is one of the two nodes that does not belong to the sub-tree containing $\Upsilon(0)$, then with probability $\frac{5}{8}$ the process terminates at $\Upsilon(n+1)$ (i.e. $\Upsilon(j)=\Upsilon(n+1)\ \forall\ j\geq n+1$).  Next $\Upsilon$ is used to define the sequence $\{ t_n\}$ in the following way: Let $t_0 =0$ and, for $k\geq 0$, let $s_k=\text{sup}\{ s\geq t_k:\Upsilon(s)=\Upsilon(t_k)\}$.  If $s_k<\infty$ let $t_{k+1}=s_k+1$.  Otherwise, let $t_{k+1}=t_k$ (note that, modulo a set of measure 0, $s_k$ only equals infinity in the case where $\Upsilon$ is stopped at one of the children of the root as described above).

We now use $\Upsilon$ and $\{ t_k\}$ to define the new process $\Phi:\mathbb{N}\rightarrow\mathbb{T}_{3,2}$ as follows: First, if $\Upsilon(0)=\varnothing$ then we just let $\Phi(k)=\Upsilon(t_k)\ \forall\ k\geq 0$.  Otherwise, let $\Phi(0)=\Upsilon(0)$ and for each $k\geq 0$ let $$\Phi(k+1)=\left\{\begin{array}{ll}\varnothing &\text{if }\Phi(k)=\varnothing\\ \Upsilon(t_{k+1}) &\text{otherwise }\end{array}\right.$$  Next an important result regarding the process $\Phi$ will be established.

\begin{figure}[H]
    \centering
    \begin{tikzpicture}
    
    \draw[fill=black] (0,0) circle (.4ex);
    \draw[fill=black] (-2,-1.5) circle (.4ex);
    \draw[fill=black] (0,-1.5) circle (.4ex);
    \draw[fill=black] (2,-1.5) circle (.4ex);
    \draw (0,0) -- (-2,-1.5);
    \draw (0,0) -- (0,-1.5);
    \draw (0,0) -- (2,-1.5);
    
    \draw[fill=black] (-2.5,-2.7) circle (.4ex);
    \draw[fill=black] (-1.5,-2.7) circle (.4ex);
    \draw[fill=black] (-.5,-2.7) circle (.4ex);
    \draw[fill=black] (.5,-2.7) circle (.4ex);
    \draw[fill=black] (1.5,-2.7) circle (.4ex);
    \draw[fill=black] (2.5,-2.7) circle (.4ex);
    
    \draw (-1.5,-2.7) -- (-2,-1.5) -- (-2.5,-2.7);
    \draw (-.5,-2.7) -- (0,-1.5) -- (.5,-2.7);
    \draw (1.5,-2.7) -- (2,-1.5) -- (2.5,-2.7);
    
    \draw[fill = black] (-2.75,-3.7) circle (.4ex);
    \draw[fill = black] (-2.5,-3.7) circle (.4ex);
    \draw[fill = black] (-2.25,-3.7) circle (.4ex);

    \draw (-2.5,-2.7) -- (-2.75,-3.7);
    \draw (-2.5,-2.7) -- (-2.5,-3.7);
    \draw (-2.5,-2.7) -- (-2.25,-3.7);
    
    \draw[dashed] (-2.75,-3.7) -- (-2.75,-4.45);
    \draw[dashed] (-2.5,-3.7) -- (-2.5,-4.45);
    \draw[dashed] (-2.25,-3.7) -- (-2.25,-4.45);
    
    \draw[fill = black] (-1.75,-3.7) circle (.4ex);
    \draw[fill = black] (-1.5,-3.7) circle (.4ex);
    \draw[fill = black] (-1.25,-3.7) circle (.4ex);
    
    \draw (-1.5,-2.7) -- (-1.75,-3.7);
    \draw (-1.5,-2.7) -- (-1.5,-3.7);
    \draw (-1.5,-2.7) -- (-1.25,-3.7);
    
    \draw[dashed] (-1.75,-3.7) -- (-1.75,-4.45);
    \draw[dashed] (-1.5,-3.7) -- (-1.5,-4.45);
    \draw[dashed] (-1.25,-3.7) -- (-1.25,-4.45);
    
    \draw[fill = black] (-.75,-3.7) circle (.4ex);
    \draw[fill = black] (-.5,-3.7) circle (.4ex);
    \draw[fill = black] (-.25,-3.7) circle (.4ex);
    
    \draw (-.5,-2.7) -- (-.75,-3.7);
    \draw (-.5,-2.7) -- (-.5,-3.7);
    \draw (-.5,-2.7) -- (-.25,-3.7);
    
    \draw[dashed] (-.75,-3.7) -- (-.75,-4.45);
    \draw[dashed] (-.5,-3.7) -- (-.5,-4.45);
    \draw[dashed] (-.25,-3.7) -- (-.25,-4.45);
    
    \draw[fill = black] (.25,-3.7) circle (.4ex);
    \draw[fill = black] (.5,-3.7) circle (.4ex);
    \draw[fill = black] (.75,-3.7) circle (.4ex);
    
    \draw (.5,-2.7) -- (.25,-3.7);
    \draw (.5,-2.7) -- (.5,-3.7);
    \draw (.5,-2.7) -- (.75,-3.7);
    
    \draw[dashed] (.25,-3.7) -- (.25,-4.45);
    \draw[dashed] (.5,-3.7) -- (.5,-4.45);
    \draw[dashed] (.75,-3.7) -- (.75,-4.45);
    
    \draw[fill = black] (1.25,-3.7) circle (.4ex);
    \draw[fill = black] (1.5,-3.7) circle (.4ex);
    \draw[fill = black] (1.75,-3.7) circle (.4ex);
    
    \draw (1.5,-2.7) -- (1.25,-3.7);
    \draw (1.5,-2.7) -- (1.5,-3.7);
    \draw (1.5,-2.7) -- (1.75,-3.7);
    
    \draw[dashed] (1.25,-3.7) -- (1.25,-4.45);
    \draw[dashed] (1.5,-3.7) -- (1.5,-4.45);
    \draw[dashed] (1.75,-3.7) -- (1.75,-4.45);
    
    \draw[fill = black] (2.25,-3.7) circle (.4ex);
    \draw[fill = black] (2.5,-3.7) circle (.4ex);
    \draw[fill = black] (2.75,-3.7) circle (.4ex);
    
    \draw (2.5,-2.7) -- (2.25,-3.7);
    \draw (2.5,-2.7) -- (2.5,-3.7);
    \draw (2.5,-2.7) -- (2.75,-3.7);
    
    \draw[dashed] (2.25,-3.7) -- (2.25,-4.45);
    \draw[dashed] (2.5,-3.7) -- (2.5,-4.45);
    \draw[dashed] (2.75,-3.7) -- (2.75,-4.45);

    \draw (0,0) node [anchor=south]{$\varnothing$};
    \draw (2,-1.5) node [anchor=south west]{$a$};
    \draw (2.5,-2.7) node [anchor=south west]{$b$};
    \draw (1.5,-2.7) node [anchor=south east]{$b'$};
    \draw (2.75,-3.7) node [anchor=south west]{$c$};

    \end{tikzpicture}
    \caption{The first four levels of $\mathbb{T}_{3,2}$ with relevant nodes labeled.}
    \label{fig:Model}
    
    \end{figure}
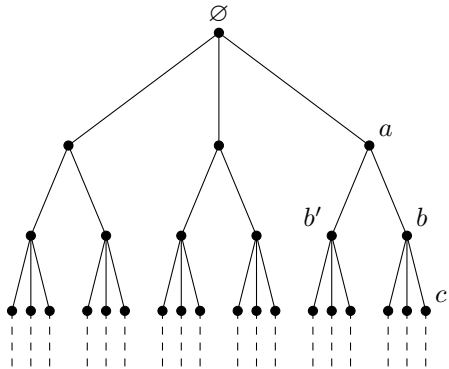
    
\begin{prop} \label{prop:bd} The process $\Phi$ is identical (in terms of its transition probabilities) to an unbiased non-backtracking random walk on $\mathbb{T}_{3,2}$ that terminates upon hitting the root.\end{prop}  
    
\begin{proof} In the case where $\Phi(0)=\varnothing$ the process $\Phi$ moves one step away from the root each time, so the conclusion follows by symmetry.  When $\Phi(0)\neq\varnothing$ a more complicated argument will be required.  We start by making some preliminary computations.  Let $p_1$ represent the probability that an unbiased random walk on $\mathbb{T}_{3,2}$ that starts at $a$ (see Figure 1 above) ever hits the root.  Likewise, let $p_2$ represent the probability that an unbiased random walk on $\mathbb{T}_{3,2}$ starting at $b$ ever hits $a$.  More generally we see by symmetry that the probability an unbiased random walk on $\mathbb{T}_{3,2}$ starting at a node on an odd numbered level (even resp.) ever hits the parent of this node is $p_1$ ($p_2$ resp.).  Calculating these values we get the expressions $$p_1=\frac{1}{3}+\frac{2}{3}p_2 p_1,\ p_2=\frac{1}{4}+\frac{3}{4}p_1 p_2\implies p_1=\frac{1}{3-2p_2},\ p_2=\frac{1}{4-3p_1}$$  Solving for $p_1$ then gives $$p_1=\frac{1}{3-\frac{2}{4-3p_1}}=\frac{4-3p_1}{10-9p_1}\implies\ 9p_1^2-13p_1+4=0\implies p_1=\frac{4}{9}\ \text{or}\ 1$$  Since the value 1 can clearly be disregarded this then gives $p_1=\frac{4}{9}$.  Plugging this into the formula for $p_2$ above we get $p_2=\frac{3}{8}$.
    
Returning now to the task of establishing that the transition probabilities of $\Phi$ match those of the non-backtracking random walk, we begin by addressing the task of showing that $\mathbb{P}(\Phi(1)=\varnothing|\Phi(0)=a)=\frac{1}{3}$.  Denoting $\mathbb{P}(\Phi(1)=\varnothing|\Phi(0)=a)$ as $p$ and $\mathbb{P}(\Upsilon(n+j)=a\ \text{for some }j>0|\Upsilon(n)=\varnothing)$ as $q$ (where we're assuming here that $\Upsilon$ originates in the sub-tree rooted at $a$), we find (based on the definition of $\Phi$) that $$p=\frac{2}{3}p_2 p+\frac{1}{3}(1-q)+\frac{1}{3}qp=\frac{1}{4}p+\frac{1}{3}(1-q)+\frac{1}{3}qp\implies p=\frac{4-4q}{9-4q}$$  Noting that $$q=\frac{1}{3}+\frac{1}{4}p_1q=\frac{1}{3}+\frac{1}{9}q\implies q=\frac{3}{8}$$ it then follows from the above formula for $p$ in terms of $q$, that indeed $p=\frac{1}{3}$.  Using this, symmetry implies that $\mathbb{P}(\Phi(1)=b|\Phi(0)=a)=\mathbb{P}(\Phi(1)=b'|\Phi(0)=a)=\frac{1}{3}$.  Hence, we find that in the case where $t=0$ and $\Phi(0)=a$, the transition probabilities of $\Phi$ do agree with those of the non-backtracking random walk that is stopped at the root.
    
Moving on, we now want to show that $\mathbb{P}(\Phi(1)=a|\Phi(0)=b)=\frac{1}{4}$.  Denoting this last probability as $p$ and the value $\mathbb{P}(\Upsilon(n+j)=b\ \text{for some }j>0|\Upsilon(n)=a)$ as $q$ (again assuming $\Upsilon$ originates in the sub-tree rooted at $a$), it follows from the definition of $\Phi$ that $$p=\frac{1}{4}(1-q)+\frac{1}{4}qp+\frac{3}{4}p_1 p=\frac{1}{4}(1-q)+\frac{1}{4}qp+\frac{1}{3}p\implies p=\frac{3-3q}{8-3q}$$  Using the fact that $$q=\frac{1}{3}+\frac{1}{3}p_2 q+\frac{1}{8}q=\frac{1}{3}+\frac{1}{4}q\implies q=\frac{4}{9}$$ our formula for $p$ in terms of $q$ then tells us that $p=\frac{1}{4}$.  Again using symmetry, we find that if $\Phi$ starts at $b$ at time $t=0$, it then goes to each of the four adjacent nodes with equal probability.  Generalizing these results, if we now let $p'_n$ represent the probability that the first step made by $\Phi$ is towards the root (given that $\Phi(0)$ resides at level n) and let $q'_n$ represent the probability that $\Upsilon$ (starting at level n-1) ever hits a particular child node of its starting node (e.g. the rightmost node), we find that it follows from induction, along with the computations for the base cases $p'_1, p'_2, q'_1,\text{ and }q'_2$ given above, that $$p'_n=\left\{\begin{array}{ll}\frac{1}{3} &\text{for n odd }\\ \frac{1}{4} &\text{for n even}\end{array}\right.$$  Once again exploiting symmetry, we find that the above result implies that when beginning at a non-root vertex, $\Phi$ moves to each of the adjacent vertices with equal probability.  
    
Now note that by the same symmetry considerations which ensure that
the transition probabilities for $\Phi$, when begun at the root, match those in
the non-backtracking case, it also follows that, following a down step, $\Phi$'s transition probabilities again match those of the
non-backtracking random walk (stopped at $\varnothing$).  Coupling this with
the results from the previous paragraph, the only remaining task
involved in establishing the proposition is addressing the case of $\Phi$'s
transition probabilities after it has just taken a step $\textit{towards}$ the root.
Since $\Phi$ always stops upon hitting the root, the case where its previous step brought it to $\varnothing$ is immediate.  Now if we let $r_n$ (for $n\geq
1$) represent the probability of $\Phi$ taking a step towards the root,
conditioned on its previous step having brought it from level $n+1$ of
$\mathbb{T}_{3,2}$ to level $n$, we find that $$r_n=\left\{\begin{array}{ll}\frac{p_2 p'_n}{p'_{n+1}} &\text{if n is odd}\\ \frac{p_1p'_n}{p'_{n+1}} &\text{if n is
even}\end{array}\right.$$Plugging in the values for $p_1,\ p_2,\text{ and }p'_n$,
then gives $r_n=\frac{1}{2}$ for n odd and $r_n=\frac{1}{3}$ for n even.  From this it then follows that, conditioned on having just moved from a node to its parent (not the root), $\Phi$ then moves to each of the available adjacent nodes (other than the one it just came from) with equal probability.  Hence, we've completed the task of showing that the transition probabilities of $\Phi$ match those of the non-backtracking random walk that is stopped at the root, and thus, have completed the proof of the proposition. 
\end{proof}
    
Having obtained the above result, the proceeding corollary regarding the non-backtracking frog model on $\mathbb{T}_{3,2}$ (see description in introduction) follows as an almost immediate consequence.
    
\begin{cor}\label{cor:jr}  There exists a coupling between the non-backtracking and original frog models on $\mathbb{T}_{3,2}$ where the path of each non-backtracking frog is a subset of the path of the corresponding frog in the original model.
\end{cor}
    
\begin{proof}  First recalling how the process $\Phi$ was constructed using $\Upsilon$, we can see that the collection of vertices landed on for an instance of $\Phi$ is a subset of the collection of vertices landed on for the corresponding instance of $\Upsilon$.  Likewise, since the process $\Upsilon$ is just a (potentially) truncated version of an unbiased random walk on $\mathbb{T}_{3,2}$, it follows from Proposition 2.1 that the non-backtracking random walk on $\mathbb{T}_{3,2}$ that terminates upon hitting the root can be coupled with the unbiased random walk on $\mathbb{T}_{3,2}$ so that the path traversed in the non-backtracking case is a subset of the path traversed in the unbiased case.  From here the entire non-backtracking frog model on $\mathbb{T}_{3,2}$ can be coupled with the original model by starting with the original, and defining a corresponding non-backtracking model for which the path of each activated frog is determined by the instance of $\Phi$ corresponding to the path traversed by the same frog in the original model.  Using this coupling, we find that the path of each frog in the non-backtracking model is a subset of that of its counterpart in the original model.
\end{proof}

\medskip
\noindent
2.2. {\bf Coupling the original and self-similar models.}  The self-similar frog model on $\mathbb{T}_{3,2}$ is obtained by refining the non-backtracking frog model through the addition of the following constraints: (i) Any frog that goes down an edge (i.e. travels away from the root) from an even to an odd level, where that edge has already been traveled along by another frog, is immediately stopped.  If multiple frogs go down a previously untraveled edge simultaneously then all but one are stopped. (ii) The same rule applies for frogs traveling down an edge from an odd to an even level $\textit{except}$ that a node on an even level can have up to two frogs land on it without being stopped (the frog originating at its parent node along with whichever frog activated the frog at its parent node) $\textit{provided}$ that the frog residing at the sibbling of the node in question has yet to be activated (see figure below).

\begin{figure}[H]
    \centering
    \begin{tikzpicture}
    
    \draw[fill=black] (0,0) circle (.4ex);
    \draw[fill=black] (-2,-1.5) circle (.4ex);
    \draw[fill=black] (0,-1.5) circle (.4ex);
    \draw[fill=black] (2,-1.5) circle (.4ex);
    \draw (0,0) -- (-2,-1.5);
    \draw (0,0) -- (0,-1.5);
    \draw (0,0) -- (2,-1.5);
    
    \draw[fill=black] (-2.5,-2.7) circle (.4ex);
    \draw[fill=black] (-1.5,-2.7) circle (.4ex);
    \draw[fill=black] (-.5,-2.7) circle (.4ex);
    \draw[fill=black] (.5,-2.7) circle (.4ex);
    \draw[fill=black] (1.5,-2.7) circle (.4ex);
    \draw[fill=black] (2.5,-2.7) circle (.4ex);
    
    \draw (-1.5,-2.7) -- (-2,-1.5) -- (-2.5,-2.7);
    \draw (-.5,-2.7) -- (0,-1.5) -- (.5,-2.7);
    \draw (1.5,-2.7) -- (2,-1.5) -- (2.5,-2.7);
    
    \draw (0,0) node [anchor=south]{$\varnothing$};
    \draw (2.3,-3.3) node [anchor=south west]{$b$};
    \draw (1.7,-3.3) node [anchor=south east]{$a$};
    
    \draw[->,line width=.7pt] (.3,0) -- (2.1,-1.35);
    \draw[->,line width=.7pt] (2.2,-1.5) -- (2.65,-2.58);
    \draw[->,line width=.7pt] (2.45,-1.38) -- (2.9,-2.46);
    
    \end{tikzpicture}
     \caption{A depiction of a scenario in which a node on an even level (node $b$) has two frogs land on it without being stopped.  Note that in order for such an event to accord with the specifications of the self-similar model, the sibling node (labeled $a$ in the figure) cannot yet have been landed on by an active frog.}
    \label{fig:Model}
    
    \end{figure}
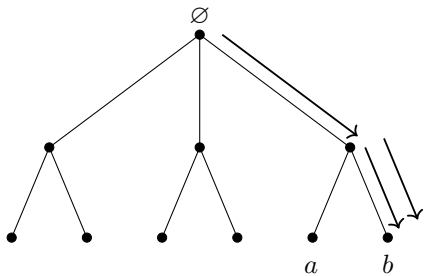
    
Since the frogs in the self-similar model defined above conduct truncated non-backtracking random walks stopped at the root, this yields a natural coupling between the self-similar and non-backtracking frog models in which the frogs in the self-similar model follow paths which are subsets of the paths followed by the corresponding non-backtracking frogs.  Composing this coupling with the coupling described in the proof of Corollary 2.2 then gives a coupling between the self-similar and original frog models that also possesses this property.  Letting $V$ and $Z$ represent the number of frogs that hit the root in the self-similar and original models respectively, we obtain the following proposition.
    
\begin{prop}\label{prop:bd1}  There exists a coupling between the self-similar and original frog models on $\mathbb{T}_{3,2}$ in which $V$ is dominated by $Z$.
\end{prop}

\noindent
Armed with this result, we now find that to prove Theorem 1.1 it suffices to prove recurrence of the self-similar frog model (i.e. that $\mathbb{P}(V=\infty)=1$).

\medskip
\noindent
2.3. {\bf Constructing the operator $\mathcal{A}$.}  Let $f(x):=\mathbb{E}[x^V]$ be the generating function for $V$.  Establishing that $\mathbb{P}(V=\infty)=1$ will involve showing that $f(x)$ is a fixed point for an operator $\mathcal{A}$.  This will be done by introducing operators $\mathcal{L}$ and $\mathcal{H}$.  We can initially think of all three operators as acting on $C^0([0, 1])$ (though we'll restrict our focus to a much smaller class of functions later on).

To start, define the random variable $V_c$ to be the number of frogs (in the self-similar frog model) originating from the sub-tree rooted at $c$ (see Figure 1 on pg. 3), which hit $b$ (conditioned on the frog at $c$ being activated).  Letting $\mathbb{T}_{3,2}(c)$ represent the sub-tree rooted at $c$, we find that if we ignore frogs originating from outside $\left\{b\right\}\cup\mathbb{T}_{3,2}(c)$ which are stopped at $b$ or $c$ after the frog at $c$ has been activated (this can be done since these frogs do not activate any other frogs in $\left\{b\right\}\cup\mathbb{T}_{3,2}(c)$), then the self-similar frog model restricted to $\left\{b\right\}\cup\mathbb{T}_{3,2}(c)$ (following the activation of the frog at $c$) looks exactly like the self-similar frog model on $\mathbb{T}_{3,2}$ following the initial step taken by the frog originating at the root.  From this it then follows that $V$ and $V_c$ have the same distribution, and therefore that $V_c$ also has $f$ as its probability generating function.

Next define the random variable $V_b$ to be the number of frogs originating from the sub-tree rooted at $b$ (see Figure 1 on pg. 3), which hit $a$ (conditioned on the frog at $b$ being activated by exactly one frog from the pair consisting of the frog starting at the root and the frog starting at $a$).  Now letting $l(x)$ represent the probability generating function of $V_b$, we present the lemma below relating the functions $l(x)$ and $f(x)$ via the following operator.

\begin{defn}\label{defn:fm1}  $\mathcal{L}g(x):=\frac{x+3}{4} g(\frac{x+2}{3})^3 +2\cdot\frac{x+2}{4} \Big(g(\frac{x+1}{3})^2-g(\frac{x+2}{3})g(\frac{x+1}{3})^2\Big)+\frac{x+1}{4}\Big(g(\frac{x}{3})-2g(\frac{x+1}{3})g(\frac{x}{3})-g(\frac{x+2}{3})^2 g(\frac{x}{3})+2g(\frac{x+2}{3})g(\frac{x+1}{3})g(\frac{x}{3})\Big)$.
\end{defn}

\begin{lemma}\label{lemma:fmt1} $l(x)=\mathcal{L}f(x)$.
\end{lemma}

\medskip
\noindent
Now the operator $\mathcal{H}$ will be introduced, along with another important lemma.  In the lemma, $h(x)$ will refer to $\mathbb{E}[x^{V'_b}]$, where $V'_b$ is the random variable representing the number of frogs originating from the sub-tree rooted at $b$ which hit $a$ (see Figures 1 and 3), conditioned on vertex $b$ being hit by both the frog that started at the root and the frog that started at $a$.

\begin{defn}\label{defn:fm2}  $\mathcal{H}g(x):=\frac{1}{3}\mathcal{L}g(x)+\frac{x+3}{6}g(\frac{x+2}{3})^3 +\frac{x+2}{6}\Big(g(\frac{x+1}{3})^2 -g(\frac{x+2}{3})g(\frac{x+1}{3})^2\Big)$.
\end{defn}

\begin{lemma}\label{lemma:fmt2} $h(x)=\mathcal{H}f(x)$
\end{lemma}

\medskip
\noindent
Next we define $\mathcal{A}$ and state the main result of this section, following which are the proofs of our two lemmas.

\begin{defn}\label{defn:fm3}  $\mathcal{A}g(x):=\frac{x}{3}\mathcal{L}[g](\frac{x}{2})+\frac{x+1}{3}\Big(\mathcal{L}[g](\frac{x+1}{2})\Big)^2-\frac{x}{3}\mathcal{L}[g](\frac{x+1}{2})\mathcal{L}[g](\frac{x}{2})+\frac{1}{3}\mathcal{H}[g](\frac{x}{2})+\frac{1}{3}\mathcal{L}[g](\frac{x+1}{2})\mathcal{H}[g](\frac{x+1}{2})-\frac{1}{3}\mathcal{L}[g](\frac{x+1}{2})\mathcal{H}[g](\frac{x}{2})$.
\end{defn}

\medskip
\noindent
$Remark\ 1.$ Note the brackets in expressions of the form $\mathcal{L}[g](\frac{x}{2})$ above, which are there to indicate that the expression is to be interpreted as the value of the function $\mathcal{L}g$ at $\frac{x}{2}$.

\begin{theorem}\label{theorem:neat2} $\mathcal{A}f=f$.
\end{theorem}

\begin{proof}[Proof of Lemma \ref{lemma:fmt1}] Observe the diagram depicting the relevant portion of $\mathbb{T}_{3,2}$ below.

\medskip
\begin{figure}[H]
    \centering
    \begin{tikzpicture}
    
     \draw[fill=black] (0,0) circle (.4ex);
    \draw[fill=black] (1.2,-1.5) circle (.4ex);
    \draw (0,0) -- (1.2,-1.5);
    
    \draw[fill=black] (.6,-3) circle (.4ex);
    \draw[fill=black] (1.2,-3) circle (.4ex);
    \draw[fill=black] (1.8,-3) circle (.4ex);
    \draw (1.2,-1.5) -- (.6,-3);
    \draw (1.2,-1.5) -- (1.2,-3);
    \draw (1.2,-1.5) -- (1.8,-3);
    
    \draw (0,.1) node [anchor=south]{$a$};
    \draw (1.35,-1.5) node [anchor=west]{$b$};
    \draw (.6,-3.07) node [anchor=north]{$c''$};
    \draw (1.2,-3.07) node [anchor=north]{$c'$};
    \draw (1.8,-3.2) node [anchor=north]{$c$};

    \end{tikzpicture}
     \caption{}
    \label{fig:Model}
    
    \end{figure}
    
\noindent
Since we are conditioning on the frog at $b$ being activated by either the frog from $a$ or the frog from the root (but not both), it follows from property (ii) of the self-similar model that no additional frogs can enter the sub-tree rooted at $b$ (meaning any such frogs are stopped at $b$).  Hence, once the frog beginning at $b$ is activated, we are starting with two active frogs there where one of them (we'll call it $\# 1$) can go in any of the four available directions and the other (call it $\#2$) must travel away from vertex $a$.  Letting A represent the event that $\# 1$ goes to $a$, $l(x)$ can then be expressed as $l(x)=\mathbb{E}[x^{V_b}]=\mathbb{E}[x^{V_b};A]+\mathbb{E}[x^{V_b};A^c]$.

\medskip
\begin{figure}[H]
    \centering
    \begin{tikzpicture}
    
       \draw[fill=black] (1.5,0) circle (.4ex);
    \draw[fill=black] (2.7,-1.5) circle (.4ex);
    \draw (1.5,0) -- (2.7,-1.5);
    
    \draw[fill=black] (2.1,-3) circle (.4ex);
    \draw[fill=black] (2.7,-3) circle (.4ex);
    \draw[fill=black] (3.3,-3) circle (.4ex);
    \draw (2.7,-1.5) -- (2.1,-3);
    \draw (2.7,-1.5) -- (2.7,-3);
    \draw (2.7,-1.5) -- (3.3,-3);
    
    \draw (1.5,.1) node [anchor=south]{$a$};
    \draw (2.95,-1.5) node [anchor=west]{$b$};
    \draw (2.1,-3.07) node [anchor=north]{$c''$};
    \draw (2.7,-3.07) node [anchor=north]{$c'$};
    \draw (3.3,-3.2) node [anchor=north]{$c$};
    
    \draw[->,line width=.7pt] (2.45,-1.5) -- (1.33,-.1);
    \draw[->,line width=.7pt] (2.9,-1.5) -- (3.48,-2.95);
    \draw[->,line width=.7pt] (3.15,-3) -- (2.81,-2.15);
    \draw[->,line width=.7pt] (2.55,-2.15) -- (2.55,-2.95);
    \draw[->,line width=.7pt] (2.44,-1.65) -- (1.92,-2.95);
    
     \draw[fill=black] (-2,0) circle (.4ex);
    \draw[fill=black] (-.8,-1.5) circle (.4ex);
    \draw (-2,0) -- (-.8,-1.5);
    
    \draw[fill=black] (-1.4,-3) circle (.4ex);
    \draw[fill=black] (-.8,-3) circle (.4ex);
    \draw[fill=black] (-.2,-3) circle (.4ex);
    \draw (-.8,-1.5) -- (-1.4,-3);
    \draw (-.8,-1.5) -- (-.8,-3);
    \draw (-.8,-1.5) -- (-.2,-3);
    
    \draw (-2,.1) node [anchor=south]{$a$};
    \draw (-.55,-1.5) node [anchor=west]{$b$};
    \draw (-1.4,-3.07) node [anchor=north]{$c''$};
    \draw (-.8,-3.07) node [anchor=north]{$c'$};
    \draw (-.2,-3.2) node [anchor=north]{$c$};
    
    \draw[->,line width=.7pt] (-1.05,-1.5) -- (-2.12,-.1);
    \draw[->,line width=.7pt] (-.6,-1.5) -- (-.02,-2.95);
    \draw[->,line width=.7pt] (-.35,-3) -- (-.69,-2.15);
    \draw[->,line width=.7pt] (-.65,-2.5) -- (-.65,-2.95);
    \draw[->,line width=.7pt] (-.95,-3) -- (-.95,-2.25);
    \draw[->,line width=.7pt] (-1.06,-1.65) -- (-1.58,-2.95);
    
     \draw[fill=black] (-5.5,0) circle (.4ex);
    \draw[fill=black] (-4.3,-1.5) circle (.4ex);
    \draw (-5.5,0) -- (-4.3,-1.5);
    
    \draw[fill=black] (-4.9,-3) circle (.4ex);
    \draw[fill=black] (-4.3,-3) circle (.4ex);
    \draw[fill=black] (-3.7,-3) circle (.4ex);
    \draw (-4.3,-1.5) -- (-4.9,-3);
    \draw (-4.3,-1.5) -- (-4.3,-3);
    \draw (-4.3,-1.5) -- (-3.7,-3);
    
    \draw (-5.5,.1) node [anchor=south]{$a$};
    \draw (-4.05,-1.5) node [anchor=west]{$b$};
    \draw (-4.9,-3.07) node [anchor=north]{$c''$};
    \draw (-4.3,-3.07) node [anchor=north]{$c'$};
    \draw (-3.7,-3.2) node [anchor=north]{$c$};
    
    \draw[->,line width=.7pt] (-4.55,-1.5) -- (-5.67,-.1);
    \draw[->,line width=.7pt] (-4.1,-1.5) -- (-3.52,-2.95);
    \draw[->,line width=.7pt] (-3.85,-3) -- (-4.19,-2.15);
    \draw[->,line width=.7pt] (-4.45,-2.15) -- (-4.45,-2.95);
    
       \draw[fill=black] (-9,0) circle (.4ex);
    \draw[fill=black] (-7.8,-1.5) circle (.4ex);
    \draw (-9,0) -- (-7.8,-1.5);
    
    \draw[fill=black] (-8.4,-3) circle (.4ex);
    \draw[fill=black] (-7.8,-3) circle (.4ex);
    \draw[fill=black] (-7.2,-3) circle (.4ex);
    \draw (-7.8,-1.5) -- (-8.4,-3);
    \draw (-7.8,-1.5) -- (-7.8,-3);
    \draw (-7.8,-1.5) -- (-7.2,-3);
    
    \draw (-9,.1) node [anchor=south]{$a$};
    \draw (-7.55,-1.5) node [anchor=west]{$b$};
    \draw (-8.4,-3.07) node [anchor=north]{$c''$};
    \draw (-7.8,-3.07) node [anchor=north]{$c'$};
    \draw (-7.2,-3.2) node [anchor=north]{$c$};
    
    \draw[->,line width=.7pt] (-8.05,-1.5) -- (-9.17,-.1);
    \draw[->,line width=.7pt] (-7.6,-1.5) -- (-7.02,-2.95);

    \end{tikzpicture}
     \caption{Illustrations representing the four events (from left to right) $A_1,\ A_2,\ A_3,$ and $A_4$.}
    \label{fig:Model}
    
    \end{figure}
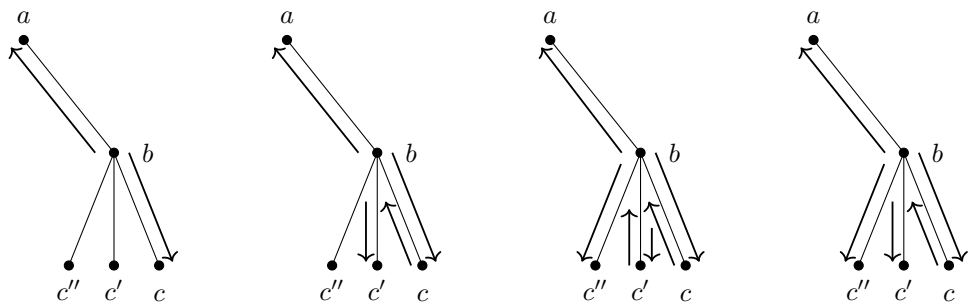
    
\noindent
Now $A$ is split up into the four separate events $A_1,\ A_2,\ A_3,$ and $A_4$ (see Figure 4 above) as follows: $A_1$ represents having the sub-tree activated by $\# 2$ fail to activate either of its two sibling sub-trees (represented by $c'$ and $c''$ in leftmost figure); $A_2$ represents the sub-tree activated by $\#2$ activating exactly one of its sibling sub-trees, which itself fails to activate the other sibling; $A_3$ represents the sub-tree activated by $\# 2$ activating exactly one of its sibling sub-trees, which itself activates the other sibling; and $A_4$ represents the sub-tree activated by $\#2$ activating both of its sibling sub-trees.
    
The next step is to evaluate $\mathbb{E}[x^{V_b};A_i]$ for each $i$ as follows:
\begin{equation}\label{A_1 comp}
\mathbb{E}[x^{V_b};A_1]=\frac{x}{4}\sum_{k=0}^{\infty}\mathbb{P}(V_c=k)\Big{(}\frac{1}{3}\Big{)}^k x^k=\frac{x}{4}f\Big{(}\frac{x}{3}\Big{)}
\end{equation}
\begin{align}\label{A_2 comp}\mathbb{E}[x^{V_b};A_2]&=\frac{x}{4}\sum_{k=1}^{\infty}\mathbb{P}(V_c=k)\sum_{j=0}^{k-1}\Big{(}\frac{1}{3}\Big{)}^j \Big{(}\frac{2}{3}\Big{)}^{k-j}\binom{k}{j}x^j\cdot 2\cdot\Big{(}\frac{1}{2}\Big{)}^{k-j}\sum_{i=0}^{\infty}\mathbb{P}(V_c=i)\Big{(}\frac{1}{3}\Big{)}^i\sum_{l=0}^i\binom{i}{l}x^l\\
&=\frac{x}{2}\sum_{k=1}^{\infty}\mathbb{P}(V_c=k)\sum_{j=0}^{k-1}\Big{(}\frac{x}{3}\Big{)}^j\Big{(}\frac{1}{3}\Big{)}^{k-j}\binom{k}{j}\sum_{i=0}^{\infty}\mathbb{P}(V_c=i)\Big{(}\frac{x+1}{3}\Big{)}^i \nonumber\\
&=\frac{x}{2}f\Big{(}\frac{x+1}{3}\Big{)}\sum_{k=1}^{\infty}\mathbb{P}(V_c=k)\sum_{j=0}^{k-1}\Big{(}\frac{x}{3}\Big{)}^j\Big{(}\frac{1}{3}\Big{)}^{k-j}\binom{k}{j}\nonumber\\
&=\frac{x}{2}f\Big{(}\frac{x+1}{3}\Big{)}\sum_{k=1}^{\infty}\mathbb{P}(V_c=k)\Big{(}\Big{(}\frac{x+1}{3}\Big{)}^k-\Big{(}\frac{x}{3}\Big{)}^k\Big{)}=\frac{x}{2}f\Big{(}\frac{x+1}{3}\Big{)}\Big{(}f\Big{(}\frac{x+1}{3}\Big{)}-f\Big{(}\frac{x}{3}\Big{)}\Big{)}\nonumber
\end{align}

\begin{align}\label{A_3 comp}
\mathbb{E}[x^{V_b};A_3]&=\frac{x}{4}\sum_{k=1}^{\infty}\mathbb{P}(V_c=k)\sum_{j=0}^{k-1}\Big{(}\frac{1}{3}\Big{)}^j\Big{(}\frac{2}{3}\Big{)}^{k-j}\binom{k}{j}x^j\cdot 2\cdot\Big{(}\frac{1}{2}\Big{)}^{k-j}\cdot\sum_{i=1}^{\infty}\mathbb{P}(V_c=i)\\
&\qquad\sum_{l=0}^{i-1}\Big{(}\frac{1}{3}\Big{)}^l\Big{(}\frac{2}{3}\Big{)}^{i-l}\binom{i}{l}x^l\Big{(}1-\Big{(}\frac{1}{2}\Big{)}^{i-l}\Big{)}\sum_{m=0}^{\infty}\mathbb{P}(V_c=m)\sum_{n=0}^m\Big{(}\frac{1}{3}\Big{)}^n\Big{(}\frac{2}{3}\Big{)}^{m-n}\binom{m}{n}x^n\nonumber\\
&=\frac{x}{2}\sum_{k=1}^{\infty}\mathbb{P}(V_c=k)\sum_{j=0}^{k-1}\Big{(}\frac{x}{3}\Big{)}^j\Big{(}\frac{1}{3}\Big{)}^{k-j}\binom{k}{j}\sum_{i=1}^{\infty}\mathbb{P}(V_c=i)\sum_{l=0}^{i-1}\Big{(}\frac{x}{3}\Big{)}^l\Big{(}\frac{2}{3}\Big{)}^{i-l}\binom{i}{l}\Big{(}1-\Big{(}\frac{1}{2}\Big{)}^{i-l}\Big{)}f\Big{(}\frac{x+2}{3}\Big{)}\nonumber\\
&=\frac{x}{2}f\Big{(}\frac{x+2}{3}\Big{)}\sum_{k=1}^{\infty}\mathbb{P}(V_c=k)\Big{(}\Big{(}\frac{x+1}{3}\Big{)}^k-\Big{(}\frac{x}{3}\Big{)}^k\Big{)}\sum_{i=1}^{\infty}\mathbb{P}(V_c=i)\Big{(}\Big{(}\frac{x+2}{3}\Big{)}^i-\Big{(}\frac{x+1}{3}\Big{)}^i\Big{)}\nonumber\\
&=\frac{x}{2}f\Big{(}\frac{x+2}{3}\Big{)}\Big{(}f\Big{(}\frac{x+1}{3}\Big{)}-f\Big{(}\frac{x}{3}\Big{)}\Big{)}\Big{(}f\Big{(}\frac{x+2}{3}\Big{)}-f\Big{(}\frac{x+1}{3}\Big{)}\Big{)}\nonumber
\end{align}
\medskip
\begin{align}\label{A_4 comp}
\mathbb{E}[x^{V_b};A_4]&=\frac{x}{4}\sum_{k=2}^{\infty}\mathbb{P}(V_c=k)\sum_{j=0}^{k-2}\Big{(}\frac{1}{3}\Big{)}^j\Big{(}\frac{2}{3}\Big{)}^{k-j}\binom{k}{j}x^j\Big{(}1-2\Big{(}\frac{1}{2}\Big{)}^{k-j}\Big{)}\\
&\cdot\Big{(}\sum_{i=0}^{\infty}\mathbb{P}(V_c=i)\sum_{l=0}^i\Big{(}\frac{1}{3}\Big{)}^l\Big{(}\frac{2}{3}\Big{)}^{i-l}\binom{i}{l}x^l\Big{)}^2\nonumber\\
&=\frac{x}{4}\sum_{k=2}^{\infty}\mathbb{P}(V_c=k)\Big{(}\Big{(}\frac{x+2}{3}\Big{)}^k-2\Big{(}\frac{x+1}{3}\Big{)}^k+\Big{(}\frac{x}{3}\Big{)}^k\Big{)}\cdot\Big{(}\sum_{i=0}^{\infty}\mathbb{P}(V_c=i)\Big{(}\frac{x+2}{3}\Big{)}^i\Big{)}^2\nonumber\\
&=\frac{x}{4}f\Big{(}\frac{x+2}{3}\Big{)}^2\Big{(}f\Big{(}\frac{x+2}{3}\Big{)}-2f\Big{(}\frac{x+1}{3}\Big{)}+f\Big{(}\frac{x}{3}\Big{)}\Big{)}\nonumber
\end{align}

\medskip
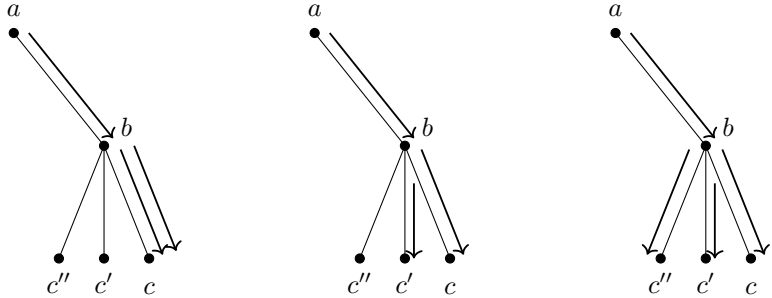
\begin{figure}[H]
    \centering
    \begin{tikzpicture}
    
       \draw[fill=black] (-3,0) circle (.4ex);
    \draw[fill=black] (-1.8,-1.5) circle (.4ex);
    \draw (-3,0) -- (-1.8,-1.5);
    
    \draw[fill=black] (-2.4,-3) circle (.4ex);
    \draw[fill=black] (-1.8,-3) circle (.4ex);
    \draw[fill=black] (-1.2,-3) circle (.4ex);
    \draw (-1.8,-1.5) -- (-2.4,-3);
    \draw (-1.8,-1.5) -- (-1.8,-3);
    \draw (-1.8,-1.5) -- (-1.2,-3);
    
    \draw (-3,.1) node [anchor=south]{$a$};
    \draw (-1.7,-1.25) node [anchor=west]{$b$};
    \draw (-2.4,-3.07) node [anchor=north]{$c''$};
    \draw (-1.8,-3.07) node [anchor=north]{$c'$};
    \draw (-1.2,-3.2) node [anchor=north]{$c$};
    
    \draw[->,line width=.7pt] (-1.58,-1.55) -- (-1.02,-2.95);
    \draw[->,line width=.7pt] (-1.4,-1.5) -- (-.84,-2.9);
    \draw[->,line width=.7pt] (-2.8,0) -- (-1.68,-1.4);
    
      \draw[fill=black] (1,0) circle (.4ex);
    \draw[fill=black] (2.2,-1.5) circle (.4ex);
    \draw (1,0) -- (2.2,-1.5);
    
    \draw[fill=black] (1.6,-3) circle (.4ex);
    \draw[fill=black] (2.2,-3) circle (.4ex);
    \draw[fill=black] (2.8,-3) circle (.4ex);
    \draw (2.2,-1.5) -- (1.6,-3);
    \draw (2.2,-1.5) -- (2.2,-3);
    \draw (2.2,-1.5) -- (2.8,-3);
    
    \draw (1,.1) node [anchor=south]{$a$};
    \draw (2.3,-1.25) node [anchor=west]{$b$};
    \draw (1.6,-3.07) node [anchor=north]{$c''$};
    \draw (2.2,-3.07) node [anchor=north]{$c'$};
    \draw (2.8,-3.2) node [anchor=north]{$c$};
    
    \draw[->,line width=.7pt] (2.42,-1.55) -- (2.98,-2.95);
    \draw[->,line width=.7pt] (1.2,0) -- (2.32,-1.4);
    \draw[->,line width=.7pt] (2.32,-2) -- (2.32,-3);
    
     \draw[fill=black] (5,0) circle (.4ex);
    \draw[fill=black] (6.2,-1.5) circle (.4ex);
    \draw (5,0) -- (6.2,-1.5);
    
    \draw[fill=black] (5.6,-3) circle (.4ex);
    \draw[fill=black] (6.2,-3) circle (.4ex);
    \draw[fill=black] (6.8,-3) circle (.4ex);
    \draw (6.2,-1.5) -- (5.6,-3);
    \draw (6.2,-1.5) -- (6.2,-3);
    \draw (6.2,-1.5) -- (6.8,-3);
    
    \draw (5,.1) node [anchor=south]{$a$};
    \draw (6.3,-1.25) node [anchor=west]{$b$};
    \draw (5.6,-3.07) node [anchor=north]{$c''$};
    \draw (6.2,-3.07) node [anchor=north]{$c'$};
    \draw (6.8,-3.2) node [anchor=north]{$c$};
    
    \draw[->,line width=.7pt] (5.98,-1.55) -- (5.42,-2.95);
    \draw[->,line width=.7pt] (6.42,-1.55) -- (6.98,-2.95);
    \draw[->,line width=.7pt] (5.2,0) -- (6.32,-1.4);
    \draw[->,line width=.7pt] (6.32,-2) -- (6.32,-3);
    
    \end{tikzpicture}
     \caption{Illustrations representing the three events (from left to right) $B_1,\ B_2$, and $B_3$.}
    \label{fig:Model}
    
    \end{figure}

Having obtained expressions for the $A_i$'s, we now split up $A^c$ into the three separate events $B_1,\ B_2$, and $B_3$ (see Figure 5 above) in the following way: $B_1$ represents having $\# 1$ and $\# 2$ activate the same sub-tree; $B_2$ represents $\# 1$ and $\# 2$ activating different sub-trees (represented by $c$ and $c'$ in middle figure above), neither of which activates the third sub-tree; and $B_3$ represents $\# 1$ and $\# 2$ activating different sub-trees, which then activate the third sub-tree.  Next the expression $\mathbb{E}[x^{V_b};B_i]$ is evaluated for each i as follows:
\begin{equation}\label{B_1 comp}
\mathbb{E}[x^{V_b};B_1]=\frac{1}{x}\mathbb{E}[x^{V_b};A]
\end{equation} (this follows from the fact that $\mathbb{P}(A)=\mathbb{P}(B_1)$ and $\big{(}V_b -1\big{)}|A=V_b|B_1$)
\begin{equation}\label{B_2 comp}
\mathbb{E}[x^{V_b};B_2]=\frac{1}{2}\Big{(}\sum_{k=0}^{\infty}\mathbb{P}(V_c=k)\sum_{j=0}^k\Big{(}\frac{1}{3}\Big{)}^k\binom{k}{j}x^j\Big{)}^2=\frac{1}{2}\Big{(}\sum_{k=0}^{\infty}\mathbb{P}(V_c=k)\Big{(}\frac{x+1}{3}\Big{)}^k\Big{)}^2=\frac{1}{2}f\Big{(}\frac{x+1}{3}\Big{)}^2
\end{equation}
\medskip
\begin{align}\label{B_3 comp}
\mathbb{E}[x^{V_b};B_3]&=\frac{1}{2}\sum_{k_1+k_2\geq 1}\mathbb{P}(V_c=k_1)\mathbb{P}(V_c=k_2)\sum_{j=0}^{k_1+k_2 -1}\Big{(}\frac{1}{3}\Big{)}^j\Big{(}\frac{2}{3}\Big{)}^{k_1+k_2 -j}x^j\binom{k_1+k_2}{j}\Big{(}1-\Big{(}\frac{1}{2}\Big{)}^{k_1+k_2-j}\Big{)}\\
&\cdot\sum_{i=0}^{\infty}\mathbb{P}(V_c=i)\sum_{l=0}^i\Big{(}\frac{1}{3}\Big{)}^l\Big{(}\frac{2}{3}\Big{)}^{i-l}\binom{i}{l}x^l\nonumber\\
&=\frac{1}{2}\sum_{k_1+k_2\geq 1}\mathbb{P}(V_c=k_1)\mathbb{P}(V_c=k_2)\Big{(}\Big{(}\frac{x+2}{3}\Big{)}^{k_1+k_2}-\Big{(}\frac{x+1}{3}\Big{)}^{k_1+k_2}\Big{)}\sum_{i=0}^{\infty}\mathbb{P}(V_c=i)\Big{(}\frac{x+2}{3}\Big{)}^i\nonumber\\
&=\frac{1}{2}f\Big{(}\frac{x+2}{3}\Big{)}\Big{(}f\Big{(}\frac{x+2}{3}\Big{)}^2-f\Big{(}\frac{x+1}{3}\Big{)}^2\Big{)}\nonumber
\end{align}

\medskip
Using the calculations from \eqref{A_1 comp}-\eqref{B_3 comp} we now find that
\begin{equation}\label{V_b comp} l(x)=\mathbb{E}[x^{V_b}]=\sum_{i=1}^4\mathbb{E}[x^{V_b};A_i]+\sum_{i=1}^3\mathbb{E}[x^{V_b};B_i]
\end{equation}
$$=\Big{(}1+\frac{1}{x}\Big{)}\Big{(}\ \frac{x}{4}f\Big{(}\frac{x}{3}\Big{)}+\frac{x}{2}f\Big{(}\frac{x+1}{3}\Big{)}^2-\frac{x}{2}f\Big{(}\frac{x+1}{3}\Big{)}f\Big{(}\frac{x}{3}\Big{)}+\frac{x}{2}f\Big{(}\frac{x+2}{3}\Big{)}^2 f\Big(\frac{x+1}{3}\Big{)}-\frac{x}{2}f\Big{(}\frac{x+2}{3}\Big{)}^2 f\Big(\frac{x}{3}\Big{)}$$
$$-\frac{x}{2}f\Big{(}\frac{x+2}{3}\Big{)}f\Big(\frac{x+1}{3}\Big{)}^2+\frac{x}{2}f\Big{(}\frac{x+2}{3}\Big{)}f\Big(\frac{x+1}{3}\Big{)}f\Big{(}\frac{x}{3}\Big{)}+\frac{x}{4}f\Big{(}\frac{x+2}{3}\Big{)}^3-\frac{x}{2}f\Big{(}\frac{x+2}{3}\Big{)}^2 f\Big(\frac{x+1}{3}\Big{)}$$
$$+\frac{x}{4}f\Big{(}\frac{x+2}{3}\Big{)}^2 f\Big(\frac{x}{3}\Big{)}\ \Big{)}\ +\ \frac{1}{2}f\Big{(}\frac{x+1}{3}\Big{)}^2+\frac{1}{2}f\Big{(}\frac{x+2}{3}\Big{)}^3-\frac{1}{2}f\Big{(}\frac{x+2}{3}\Big{)}f\Big{(}\frac{x+1}{3}\Big{)}^2$$

$$=\frac{x+3}{4}f\Big{(}\frac{x+2}{3}\Big{)}^3+2\cdot\frac{x+2}{4}\Big{(}f\Big{(}\frac{x+1}{3}\Big{)}^2-f\Big{(}\frac{x+2}{3}\Big{)}f\Big{(}\frac{x+1}{3}\Big{)}^2\Big{)}+\frac{x+1}{4}\Big{(}f\Big{(}\frac{x}{3}\Big{)}$$
$$-2 f\Big{(}\frac{x+1}{3}\Big{)}f\Big{(}\frac{x}{3}\Big{)}-f\Big{(}\frac{x+2}{3}\Big{)}^2 f\Big{(}\frac{x}{3}\Big{)}+2f\Big{(}\frac{x+2}{3}\Big{)}f\Big{(}\frac{x+1}{3}\Big{)}f\Big{(}\frac{x}{3}\Big{)}\Big{)}=\mathcal{L}f(x)$$

\bigskip
\noindent
Hence, the proof of Lemma $\ref{lemma:fmt1}$ is complete.
\end{proof}

\begin{proof}[Proof of Lemma \ref{lemma:fmt2}] The scenario under consideration (see Figure 3 again) begins with three active frogs at vertex $b$, where one (call it $\# 1$) is free to go in any of the four available directions, and the other two (call them the $\# 2$ frogs) can go in any of the three directions away from the root.  Letting $A_0$ represent the event that the two $\#2$ frogs travel to the same node from $b$ (call it $c$), $h(x)$ can be expressed as $\mathbb{E}[x^{V'_b}]=\mathbb{E}[x^{V'_b};A_0]+\mathbb{E}[x^{V'_b};A^c_0]$.  In the event $A_0$, since one of the two $\# 2$ frogs is stopped at $c$, it follows that $V'_b|A_0$ has the same distribution as $V_b$.  Hence, $\mathbb{E}[x^{V'_b};A_0]=\mathbb{P}(A_0)\mathbb{E}[x^{V'_b}|A_0]=\mathbb{P}(A_0)\mathbb{E}[x^{V_b}]=\frac{1}{3}\mathcal{L}f(x)$.

\medskip
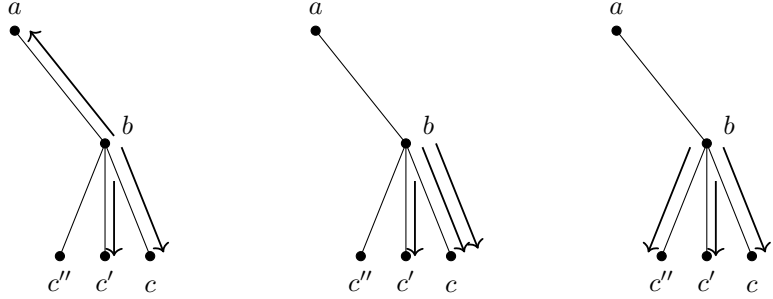
\begin{figure}[H]
    \centering
    \begin{tikzpicture}
    
    \draw[fill=black] (-3,0) circle (.4ex);
    \draw[fill=black] (-1.8,-1.5) circle (.4ex);
    \draw (-3,0) -- (-1.8,-1.5);
    
    \draw[fill=black] (-2.4,-3) circle (.4ex);
    \draw[fill=black] (-1.8,-3) circle (.4ex);
    \draw[fill=black] (-1.2,-3) circle (.4ex);
    \draw (-1.8,-1.5) -- (-2.4,-3);
    \draw (-1.8,-1.5) -- (-1.8,-3);
    \draw (-1.8,-1.5) -- (-1.2,-3);
    
    \draw (-3,.1) node [anchor=south]{$a$};
    \draw (-1.7,-1.25) node [anchor=west]{$b$};
    \draw (-2.4,-3.07) node [anchor=north]{$c''$};
    \draw (-1.8,-3.07) node [anchor=north]{$c'$};
    \draw (-1.2,-3.2) node [anchor=north]{$c$};
    
    \draw[->,line width=.7pt] (-1.58,-1.55) -- (-1.02,-2.95);
    \draw[->,line width=.7pt] (-1.68,-1.4) -- (-2.8,0);
    \draw[->,line width=.7pt] (-1.68,-2) -- (-1.68,-3);
    
      \draw[fill=black] (1,0) circle (.4ex);
    \draw[fill=black] (2.2,-1.5) circle (.4ex);
    \draw (1,0) -- (2.2,-1.5);
    
    \draw[fill=black] (1.6,-3) circle (.4ex);
    \draw[fill=black] (2.2,-3) circle (.4ex);
    \draw[fill=black] (2.8,-3) circle (.4ex);
    \draw (2.2,-1.5) -- (1.6,-3);
    \draw (2.2,-1.5) -- (2.2,-3);
    \draw (2.2,-1.5) -- (2.8,-3);
    
    \draw (1,.1) node [anchor=south]{$a$};
    \draw (2.3,-1.25) node [anchor=west]{$b$};
    \draw (1.6,-3.07) node [anchor=north]{$c''$};
    \draw (2.2,-3.07) node [anchor=north]{$c'$};
    \draw (2.8,-3.2) node [anchor=north]{$c$};
    
    \draw[->,line width=.7pt] (2.42,-1.55) -- (2.98,-2.95);
    \draw[->,line width=.7pt] (2.6,-1.5) -- (3.16,-2.9);
    \draw[->,line width=.7pt] (2.32,-2) -- (2.32,-3);
    
     \draw[fill=black] (5,0) circle (.4ex);
    \draw[fill=black] (6.2,-1.5) circle (.4ex);
    \draw (5,0) -- (6.2,-1.5);
    
    \draw[fill=black] (5.6,-3) circle (.4ex);
    \draw[fill=black] (6.2,-3) circle (.4ex);
    \draw[fill=black] (6.8,-3) circle (.4ex);
    \draw (6.2,-1.5) -- (5.6,-3);
    \draw (6.2,-1.5) -- (6.2,-3);
    \draw (6.2,-1.5) -- (6.8,-3);
    
    \draw (5,.1) node [anchor=south]{$a$};
    \draw (6.3,-1.25) node [anchor=west]{$b$};
    \draw (5.6,-3.07) node [anchor=north]{$c''$};
    \draw (6.2,-3.07) node [anchor=north]{$c'$};
    \draw (6.8,-3.2) node [anchor=north]{$c$};
    
    \draw[->,line width=.7pt] (5.98,-1.55) -- (5.42,-2.95);
    \draw[->,line width=.7pt] (6.42,-1.55) -- (6.98,-2.95);
    \draw[->,line width=.7pt] (6.32,-2) -- (6.32,-3);
    
      \end{tikzpicture}
     \caption{Illustrations representing the three events (from left to right) $C_1,\ C_2$, and $C_3$.}
    \label{fig:Model}
    
    \end{figure}

Turning next to the event $A^c_0$, it will be split up into the events $C_1,\ C_2,$ and $C_3$ (see Figure 6 above) as follows: $C_1$ represents having the $\#2$ frogs go to different nodes and the $\#1$ frog go to $a$; $C_2$ represents the $\#2$ frogs going to different nodes and the $\#1$ frog going to the same node as one of the $\#2$ frogs; and $C_3$ represents the $\#2$ frogs going to different nodes and the $\#1$ frog going to the third sibling node.  Evaluating $\mathbb{E}[x^{V'_b};C_i]$ for each i now gives the following:
\begin{align}\label{C_1 comp}
\mathbb{E}[x^{V'_b};C_1]&=\frac{x}{4}\cdot\frac{2}{3}\cdot\mathbb{E}[x^{V_b}|B_2\cup B_3]=\frac{x}{6}\cdot\frac{\Big{(}\frac{1}{2}f\Big{(}\frac{x+1}{3}\Big{)}^2 +\frac{1}{2}f\Big{(}\frac{x+2}{3}\Big{)}^3-\frac{1}{2}f\Big{(}\frac{x+2}{3}\Big{)}f\Big{(}\frac{x+1}{3}\Big{)}^2\Big{)}}{1/2}\\
&=\frac{x}{6}\Big{(}f\Big{(}\frac{x+1}{3}\Big{)}^2 +f\Big{(}\frac{x+2}{3}\Big{)}^3-f\Big{(}\frac{x+2}{3}\Big{)}f\Big{(}\frac{x+1}{3}\Big{)}^2\Big{)}\nonumber
\end{align}
\medskip
\begin{equation}\label{C_2 comp}
\mathbb{E}[x^{V'_b};C_2]=\frac{1}{3}\mathbb{E}[x^{V_b}|B_2\cup B_3]=\frac{1}{3}\Big{(}f\Big{(}\frac{x+1}{3}\Big{)}^2 +f\Big{(}\frac{x+2}{3}\Big{)}^3-f\Big{(}\frac{x+2}{3}\Big{)}f\Big{(}\frac{x+1}{3}\Big{)}^2\Big{)}
\end{equation}
\medskip
\begin{equation}\label{C_3 comp}
\mathbb{E}[x^{V'_b};C_3]=\frac{1}{6}\Big{(}\sum_{k=0}^{\infty}\mathbb{P}(V_c=k)\sum_{j=0}^k \Big{(}\frac{1}{3}\Big{)}^j\Big{(}\frac{2}{3}\Big{)}^{k-j}\binom{k}{j}x^j\Big{)}^3=\frac{1}{6}f\Big{(}\frac{x+2}{3}\Big{)}^3
\end{equation}
Adding the expressions \eqref{C_1 comp}-\eqref{C_3 comp} to our expression for $\mathbb{E}[x^{V'_b};A_0]$ then gives
$$h(x)=\mathbb{E}[x^{V'_b}]=\frac{1}{3}\mathcal{L}f(x)+\frac{x+3}{6}f\Big{(}\frac{x+2}{3}\Big{)}^3 +\frac{x+2}{6}\Big(f\Big(\frac{x+1}{3}\Big)^2 -f\Big(\frac{x+2}{3}\Big)f\Big(\frac{x+1}{3}\Big)^2\Big)=\mathcal{H}f(x)$$
Hence, the proof is complete.
\end{proof}

With Lemmas \ref{lemma:fmt1} and \ref{lemma:fmt2} established, the proof of Theorem \ref{theorem:neat2} can now be presented.

\begin{proof}[Proof of Theorem \ref{theorem:neat2}] Begin by separating the collection of possible outcomes into the three events $D_1,\ D_2,$ and $D_3$ (see Figures 7, 8, and 9 below).

\bigskip
\begin{figure}[H]
    \centering
    \begin{tikzpicture}
    
    \draw[fill=black] (0,0) circle (.4ex);
    \draw[fill=black] (-2,-1.5) circle (.4ex);
    \draw[fill=black] (0,-1.5) circle (.4ex);
    \draw[fill=black] (2,-1.5) circle (.4ex);
    \draw (0,0) -- (-2,-1.5);
    \draw (0,0) -- (0,-1.5);
    \draw (0,0) -- (2,-1.5);
    
    \draw[fill=black] (-2.5,-2.7) circle (.4ex);
    \draw[fill=black] (-1.5,-2.7) circle (.4ex);
    \draw[fill=black] (-.5,-2.7) circle (.4ex);
    \draw[fill=black] (.5,-2.7) circle (.4ex);
    \draw[fill=black] (1.5,-2.7) circle (.4ex);
    \draw[fill=black] (2.5,-2.7) circle (.4ex);
    
    \draw (-1.5,-2.7) -- (-2,-1.5) -- (-2.5,-2.7);
    \draw (-.5,-2.7) -- (0,-1.5) -- (.5,-2.7);
    \draw (1.5,-2.7) -- (2,-1.5) -- (2.5,-2.7);
    
    \draw (0,0) node [anchor=south]{$\varnothing$};
    \draw (2.55,-1.5) node [anchor=south east]{$a$};
    \draw (2.75,-3.3) node [anchor=south east]{$b$};
    \draw (1.75,-3.3) node [anchor=south east]{$b'$};

    \draw[->,line width=.7pt] (.3,0) -- (2.1,-1.35);
    \draw[->,line width=.7pt] (2.2,-1.5) -- (2.65,-2.58);
    \draw[->,line width=.7pt] (1.8,-1.5) -- (1.35,-2.58);
    
    \end{tikzpicture}
     \caption{A representation of $D_1$, defined as the event in which the frog coming from the root and the  frog at the first vertex it hits (labelled $a$ in the figure above) go to different children of $a$.}
    \label{fig:Model}
    
    \end{figure}
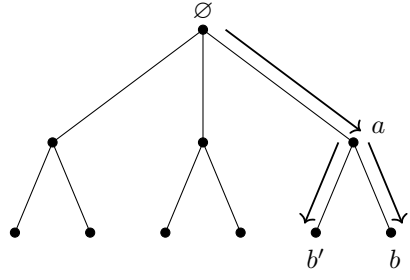
    
    \bigskip
    \begin{figure}[H]
    \centering
    \begin{tikzpicture}
    
    \draw[fill=black] (0,0) circle (.4ex);
    \draw[fill=black] (-2,-1.5) circle (.4ex);
    \draw[fill=black] (0,-1.5) circle (.4ex);
    \draw[fill=black] (2,-1.5) circle (.4ex);
    \draw (0,0) -- (-2,-1.5);
    \draw (0,0) -- (0,-1.5);
    \draw (0,0) -- (2,-1.5);
    
    \draw[fill=black] (-2.5,-2.7) circle (.4ex);
    \draw[fill=black] (-1.5,-2.7) circle (.4ex);
    \draw[fill=black] (-.5,-2.7) circle (.4ex);
    \draw[fill=black] (.5,-2.7) circle (.4ex);
    \draw[fill=black] (1.5,-2.7) circle (.4ex);
    \draw[fill=black] (2.5,-2.7) circle (.4ex);
    
    \draw (-1.5,-2.7) -- (-2,-1.5) -- (-2.5,-2.7);
    \draw (-.5,-2.7) -- (0,-1.5) -- (.5,-2.7);
    \draw (1.5,-2.7) -- (2,-1.5) -- (2.5,-2.7);
    
    \draw (0,0) node [anchor=south]{$\varnothing$};
    \draw (2.55,-1.5) node [anchor=south east]{$a$};
    \draw (2.75,-3.3) node [anchor=south east]{$b$};
    \draw (1.75,-3.3) node [anchor=south east]{$b'$};
    
    \draw[->,line width=.7pt] (.3,0) -- (2.1,-1.35);
    \draw[->,line width=.7pt] (1.75,-1.54) -- (.1,-.3);
    \draw[->,line width=.7pt] (2.2,-1.5) -- (2.65,-2.58);
    
    \end{tikzpicture}
     \caption{A representation of $D_2$, defined as the event in which the frog at the first vertex hit, upon being activated, returns to the root.}
    \label{fig:Model}
    
    \end{figure}
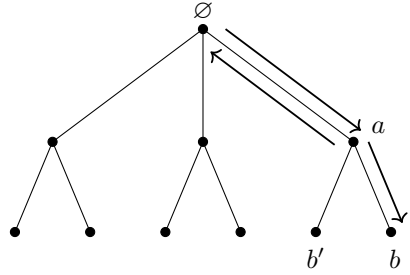

\bigskip
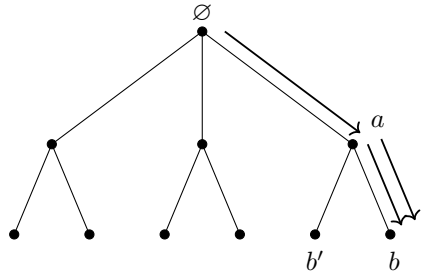
\begin{figure}[H]
    \centering
    \begin{tikzpicture}
    
    \draw[fill=black] (0,0) circle (.4ex);
    \draw[fill=black] (-2,-1.5) circle (.4ex);
    \draw[fill=black] (0,-1.5) circle (.4ex);
    \draw[fill=black] (2,-1.5) circle (.4ex);
    \draw (0,0) -- (-2,-1.5);
    \draw (0,0) -- (0,-1.5);
    \draw (0,0) -- (2,-1.5);
    
    \draw[fill=black] (-2.5,-2.7) circle (.4ex);
    \draw[fill=black] (-1.5,-2.7) circle (.4ex);
    \draw[fill=black] (-.5,-2.7) circle (.4ex);
    \draw[fill=black] (.5,-2.7) circle (.4ex);
    \draw[fill=black] (1.5,-2.7) circle (.4ex);
    \draw[fill=black] (2.5,-2.7) circle (.4ex);
    
    \draw (-1.5,-2.7) -- (-2,-1.5) -- (-2.5,-2.7);
    \draw (-.5,-2.7) -- (0,-1.5) -- (.5,-2.7);
    \draw (1.5,-2.7) -- (2,-1.5) -- (2.5,-2.7);
    
    \draw (0,0) node [anchor=south]{$\varnothing$};
    \draw (2.55,-1.4) node [anchor=south east]{$a$};
    \draw (2.75,-3.3) node [anchor=south east]{$b$};
    \draw (1.75,-3.3) node [anchor=south east]{$b'$};
    
    \draw[->,line width=.7pt] (.3,0) -- (2.1,-1.35);
    \draw[->,line width=.7pt] (2.2,-1.5) -- (2.65,-2.58);
    \draw[->,line width=.7pt] (2.38,-1.43) -- (2.83,-2.51);
    
    \end{tikzpicture}
     \caption{A representation of $D_3$, defined as the event in which the frog coming from the root and the frog coming from $a$ (where $a$ once again represents the first vertex landed on) go to the same child of $a$.}
    \label{fig:Model}
    
    \end{figure}
    
\medskip
\noindent
Next we compute $\mathbb{E}[x^V;D_i]$ for each i beginning with $i=1$.
\begin{equation}\label{D_1 comp}
\mathbb{E}[x^{V};D_1]=\frac{1}{3}\Big{(}\sum_{k=0}^{\infty}\mathbb{P}(V_b=k)\sum_{j=0}^k\Big{(}\frac{1}{2}\Big{)}^kx^j\binom{k}{j}\Big{)}^2=\frac{1}{3}\Big{(}\mathcal{L}[f]\Big{(}\frac{x+1}{2}\Big{)}\Big{)}^2
\end{equation}
(where above we use the fact, shown in \eqref{V_b comp}, that $\mathbb{E}[x^{V_b}]=\mathcal{L}f(x)$).  $D_2$ can be separated into the two events $D^{(1)}_2$ and $D^{(2)}_2$ as follows: $D^{(1)}_2$ represents having all frogs that go to $a$ from the sub-tree rooted at $b$ then travel to the root; and $D^{(2)}_2$ represents having at least one frog that travels to $a$ from the sub-tree rooted at $b$ then go to $b'$ (i.e. $D_2/D^{(1)}_2$).  Computing $\mathbb{E}[x^V;D^{(i)}_2]$ for $i=1,2$ now gives
\begin{equation}\label{D_2 comp1}
\mathbb{E}[x^V;D^{(1)}_2]=\frac{x}{3}\sum_{k=0}^{\infty}\mathbb{P}(V_b=k)\Big{(}\frac{1}{2}\Big{)}^k x^k=\frac{x}{3}\mathcal{L}[f]\Big{(}\frac{x}{2}\Big{)}
\end{equation}
\begin{align}\label{D_2 comp2}
\mathbb{E}[x^V;D^{(2)}_2]&=\frac{x}{3}\sum_{k=1}^{\infty}\mathbb{P}(V_b=k)\sum_{j=0}^{k-1}\Big{(}\frac{1}{2}\Big{)}^k\binom{k}{j}x^j\mathcal{L}[f]\Big{(}\frac{x+1}{2}\Big{)}\\
&=\frac{x}{3}\mathcal{L}[f]\Big{(}\frac{x+1}{2}\Big{)}\sum_{k=1}^{\infty}\mathbb{P}(V_b=k)\Big{(}\Big{(}\frac{x+1}{2}\Big{)}^k-\Big{(}\frac{x}{2}\Big{)}^k\Big{)}\nonumber\\
&=\frac{x}{3}\mathcal{L}[f]\Big{(}\frac{x+1}{2}\Big{)}\Big{(}\mathcal{L}[f]\Big{(}\frac{x+1}{2}\Big{)}-\mathcal{L}[f]\Big{(}\frac{x}{2}\Big{)}\Big{)}\nonumber
\end{align}
Moving on to $D_3$, it can also be broken up into two separate events in the following way: $D^{(1)}_3$ represents having all frogs that go to $a$ from the sub-tree rooted at $b$ then travel to the root; and $D^{(2)}_3$ represents having at least one frog that travels to $a$ from the sub-tree rooted at $b$ then go to $b'$ (note the only difference between these two events and the events $D^{(1)}_2$ and $D^{(2)}_2$ respectively is the behavior of the frog starting at $a$; as seen in Figures 8 and 9).  Computing $\mathbb{E}[x^V;D^{(i)}_3]$ for $i=1,2$ gives
\begin{equation}\label{D_3 comp1}
\mathbb{E}[x^V;D^{(1)}_3]=\frac{1}{3}\sum_{k=0}^{\infty}\mathbb{P}(V'_b=k)\Big{(}\frac{1}{2}\Big{)}^k x^k=\frac{1}{3}\mathcal{H}[f]\Big{(}\frac{x}{2}\Big{)}
\end{equation}
\begin{align}\label{D_3 comp2}
\mathbb{E}[x^V;D^{(2)}_3]&=\frac{1}{3}\sum_{k=1}^{\infty}\mathbb{P}(V'_b=k)\sum_{j=0}^{k-1}\Big{(}\frac{1}{2}\Big{)}^k\binom{k}{j}x^j\mathcal{L}[f]\Big{(}\frac{x+1}{2}\Big{)}\\
&=\frac{1}{3}\mathcal{L}[f]\Big{(}\frac{x+1}{2}\Big{)}\sum_{k=1}^{\infty}\mathbb{P}(V'_b=k)\Big{(}\Big{(}\frac{x+1}{2}\Big{)}^k-\Big{(}\frac{x}{2}\Big{)}^k\Big{)}\nonumber\\
&=\frac{1}{3}\mathcal{L}[f]\Big{(}\frac{x+1}{2}\Big{)}\Big{(}\mathcal{H}[f]\Big{(}\frac{x+1}{2}\Big{)}-\mathcal{H}[f]\Big{(}\frac{x}{2}\Big{)}\Big{)}\nonumber
\end{align}
Now adding together the expressions \eqref{D_1 comp}-\eqref{D_3 comp2} gives
\begin{align}
f(x)&=\mathbb{E}[x^V]=\sum_{i=1}^3\mathbb{E}[x^V;D_i]\nonumber\\
&=\frac{1}{3}\Big{(}\mathcal{L}[f]\Big{(}\frac{x+1}{2}\Big{)}\Big{)}^2+\frac{x}{3}\mathcal{L}[f]\Big{(}\frac{x}{2}\Big{)}
+\frac{x}{3}\mathcal{L}[f]\Big{(}\frac{x+1}{2}\Big{)}\Big{(}\mathcal{L}[f]\Big{(}\frac{x+1}{2}\Big{)}-\mathcal{L}[f]\Big{(}\frac{x}{2}\Big{)}\Big{)}\nonumber\\
&+\frac{1}{3}\mathcal{H}[f]\Big{(}\frac{x}{2}\Big{)}+\frac{1}{3}\mathcal{L}[f]\Big{(}\frac{x+1}{2}\Big{)}\Big{(}\mathcal{H}[f]\Big{(}\frac{x+1}{2}\Big{)}-\mathcal{H}[f]\Big{(}\frac{x}{2}\Big{)}\Big{)}\nonumber\\
&=\frac{x}{3}\mathcal{L}[f]\Big(\frac{x}{2}\Big)+\frac{x+1}{3}\Big(\mathcal{L}[f]\Big(\frac{x+1}{2}\Big)\Big)^2-\frac{x}{3}\mathcal{L}[f]\Big(\frac{x+1}{2}\Big)\mathcal{L}[f]\Big(\frac{x}{2}\Big)+\frac{1}{3}\mathcal{H}[f]\Big(\frac{x}{2}\Big)\nonumber\\
&+\frac{1}{3}\mathcal{L}[f]\Big(\frac{x+1}{2}\Big)\mathcal{H}[f]\Big(\frac{x+1}{2}\Big)-\frac{1}{3}\mathcal{L}[f]\Big(\frac{x+1}{2}\Big)\mathcal{H}[f]\Big(\frac{x}{2}\Big)=\mathcal{A}f(x)\nonumber
\end{align}
Hence, the proof of Theorem \ref{theorem:neat2} is complete.
\end{proof}

\medskip
\noindent
2.4. {\bf Monotonicity of $\mathcal{A}$.} In order to prove Theorem \ref{theorem:neat1} (i.e. show that $\mathbb{P}(V=\infty)=1$) it suffices to show that $f(x)=0$ on $[0, 1)$.  With the proof of Theorem \ref{theorem:neat2} now complete, this task is reduced to showing that $\mathcal{A}^n f(x)\rightarrow 0$ as $n\rightarrow\infty\ \forall\ x\in [0, 1)$.  The first major step involved in accomplishing this will be to prove the following proposition.

\begin{prop}\label{prop:monotonicity}
Define $\mathcal{S}$ to be the space of all probability generating functions (on $[0, 1]$) associated with probability distributions on $\left\{0,1,\dots\right\}\cup\left\{\infty\right\}$.  Let $g_1, g_2\in\mathcal{S}$ with $g_1\geq g_2$ on $[0, 1]$.  Then $\mathcal{A}g_1\geq\mathcal{A}g_2$ on $[0, 1]$.
\end{prop}

\medskip
\noindent
The proof of \ref{prop:monotonicity} will require the lemma below.

\begin{lemma}\label{lemma:S closed} $\mathcal{LS}\subseteq\mathcal{S}$, $\mathcal{HS}\subseteq\mathcal{S}$, and $\mathcal{AS}\subseteq\mathcal{S}$.
\end{lemma}

\begin{proof}
Begin by defining the following model: Start with a single active frog at the root and sleeping frogs at the other three nodes (see Figure 10 below).  The frog at the root performs a non-backtracking random walk that is stopped upon hitting any one of the six boxes, and any time an active frog hits a vertex with a sleeping frog, that frog is activated and begins performing its own non-backtracking random walk that is stopped upon hitting either the root or one of the boxes.  In addition, the first time a box is hit by a frog, it releases frogs which also perform non-backtracking random walks that are stopped upon hitting either the root or another box.

\bigskip
\begin{figure}[H]
    \centering
    \begin{tikzpicture}
    
    \draw[fill=black] (0,0) circle (.4ex);
    \draw[fill=black] (2,-1.5) circle (.4ex);
    \draw (0,0) -- (2,-1.5);
    
    \draw[fill=black] (.9,-2.9) circle (.4ex);
    \draw[fill=black] (3.1,-2.9) circle (.4ex);
    
    \draw (.9,-2.9) -- (2,-1.5) -- (3.1,-2.9);
    
    \draw (.9,-2.9) -- (.2,-4.1);
    \draw (.9,-2.9) -- (.9,-4.1);
    \draw (.9,-2.9) -- (1.6,-4.1);
    
    \draw (3.1,-2.9) -- (2.4,-4.1);
    \draw (3.1,-2.9) -- (3.1,-4.1);
    \draw (3.1,-2.9) -- (3.8,-4.1);
    
    \draw (-.05,-4.1) -- (.45,-4.1);
    \draw (-.05,-4.1) -- (-.05,-4.6);
    \draw (.45,-4.1) -- (.45,-4.6);
    \draw (-.05,-4.6) -- (.45,-4.6);
    
    \draw (.65,-4.1) -- (1.15,-4.1);
    \draw (.65,-4.1) -- (.65,-4.6);
    \draw (1.15,-4.1) -- (1.15,-4.6);
    \draw (.65,-4.6) -- (1.15,-4.6);
    
    \draw (1.35,-4.1) -- (1.85,-4.1);
    \draw (1.35,-4.1) -- (1.35,-4.6);
    \draw (1.85,-4.1) -- (1.85,-4.6);
    \draw (1.35,-4.6) -- (1.85,-4.6);
    
    \draw (2.15,-4.1) -- (2.65,-4.1);
    \draw (2.15,-4.1) -- (2.15,-4.6);
    \draw (2.65,-4.1) -- (2.65,-4.6);
    \draw (2.15,-4.6) -- (2.65,-4.6);
    
    \draw (2.85,-4.1) -- (3.35,-4.1);
    \draw (2.85,-4.1) -- (2.85,-4.6);
    \draw (3.35,-4.1) -- (3.35,-4.6);
    \draw (2.85,-4.6) -- (3.35,-4.6);
    
    \draw (3.55,-4.1) -- (4.05,-4.1);
    \draw (3.55,-4.1) -- (3.55,-4.6);
    \draw (4.05,-4.1) -- (4.05,-4.6);
    \draw (3.55,-4.6) -- (4.05,-4.6);
    
    \draw (0,0) node [anchor=south]{$\varnothing$};
    \draw (3.5,-2.9) node [anchor=south east]{$b$};
    \draw (1,-2.9) node [anchor=south east]{$b'$};
    \draw (.48,-4.58) node [anchor=south east]{$U$};
    \draw (1.18,-4.58) node [anchor=south east]{$U$};
    \draw (1.88,-4.58) node [anchor=south east]{$U$};
    \draw (2.68,-4.58) node [anchor=south east]{$U$};
    \draw (3.38,-4.58) node [anchor=south east]{$U$};
    \draw (4.08,-4.58) node [anchor=south east]{$U$};
    
    \end{tikzpicture}
     \caption{}
    \label{fig:Model}
    
    \end{figure}
    
\medskip
\noindent
The number of frogs released by the different boxes, conditioned on being hit, are i.i.d. random variables with distribution $U$.  Finally, the model obeys property (ii) with respect to the nodes $b$ and $b'$ (see beginning of Section 2.2).  Now let $\mathcal{A}^*U$ represent the distribution of the number of frogs that hit the root in this model.  It then follows that $\mathcal{A}^*\tilde{V}_c=\tilde{V}$ (where $\tilde{V}$ and $\tilde{V}_c$ represent the distributions of $V$ and $V_c$).  Now recall that the proof of Theorem \ref{theorem:neat2} involved calculating the generating function of $V$ (denoted as $f(x)$) in terms of the generating function of $V_c$ (also denoted as $f(x)$ on account of our recognition that $V$ and $V_c$ share the same distribution) and showing that $V$ has generating function $\mathcal{A}f$ (i.e. $\mathcal{A}f$ is the generating function associated with the distribution $\mathcal{A}^*\tilde{V}_c$).  Since the derivation of this formula was carried out purely symbolically (meaning without taking into account the particular properties of $V_c$ or its generating function $f$), this means that for any probability distribution $U$ (concentrated on $\left\{1,2,\dots\right\}\cup\left\{\infty\right\}$) with generating function $\eta$, the generating function of the distribution $\mathcal{A}^*U$ is $\mathcal{A}\eta$.  Hence, it follows that $\mathcal{AS}\subseteq\mathcal{S}$.

The proofs of $\mathcal{LS}\subseteq\mathcal{S}$ and $\mathcal{HS}\subseteq\mathcal{S}$ are very similar to the proof of $\mathcal{AS}\subseteq\mathcal{S}$, so some of the details will therefore be omitted.  In both cases we define a model using the diagram below (see Figure 11).  For $\mathcal{L}$, begin with two active frogs at vertex $b$, one of which must go in one of the three downward directions, while the other is free to go in any of the four available directions.  Active frogs are to perform non-backtracking random walks which stop upon hitting either $a$ or any of the boxes.  The first time a box is hit by an active frog, it releases active frogs according to the distribution $U$.  The numbers of frogs released by the different boxes (conditioned on being hit) are independent.

\bigskip
\begin{figure}[H]
    \centering
    \begin{tikzpicture}
    
    \draw[fill=black] (1.45,-1.36) circle (.4ex);
    \draw[fill=black] (3.1,-2.9) circle (.4ex);
    \draw (1.45,-1.36) -- (3.1,-2.9);
    
    \draw (3.1,-2.9) -- (2.4,-4.1);
    \draw (3.1,-2.9) -- (3.1,-4.1);
    \draw (3.1,-2.9) -- (3.8,-4.1);
    
    \draw (2.15,-4.1) -- (2.65,-4.1);
    \draw (2.15,-4.1) -- (2.15,-4.6);
    \draw (2.65,-4.1) -- (2.65,-4.6);
    \draw (2.15,-4.6) -- (2.65,-4.6);
    
    \draw (2.85,-4.1) -- (3.35,-4.1);
    \draw (2.85,-4.1) -- (2.85,-4.6);
    \draw (3.35,-4.1) -- (3.35,-4.6);
    \draw (2.85,-4.6) -- (3.35,-4.6);
    
    \draw (3.55,-4.1) -- (4.05,-4.1);
    \draw (3.55,-4.1) -- (3.55,-4.6);
    \draw (4.05,-4.1) -- (4.05,-4.6);
    \draw (3.55,-4.6) -- (4.05,-4.6);
    
    \draw (1.65,-1.25) node [anchor=south east]{$a$};
    \draw (3.5,-2.9) node [anchor=south east]{$b$};
    \draw (2.68,-4.58) node [anchor=south east]{$U$};
    \draw (3.38,-4.58) node [anchor=south east]{$U$};
    \draw (4.08,-4.58) node [anchor=south east]{$U$};
    
    \end{tikzpicture}
     \caption{}
    \label{fig:Model}
    
    \end{figure}
    
\medskip
\noindent
Letting $\mathcal{L}^*U$ represent the distribution of the number of frogs that hit $a$, we find (by a similar argument to the one used for $\mathcal{A}^*$) that the generating function of $\mathcal{L}^*U$ is $\mathcal{L}\eta$ (where $\eta$ once again represents the generating function associated with the distribution $U$).  From this it follows that $\mathcal{LS}\subseteq\mathcal{S}$.  Furthermore, using a model which differs from this one only in that a single additional active frog that can go in any of the three downward directions is positioned at $b$, we also find that $\mathcal{H}^*U$ has generating function $\mathcal{H}\eta$, from which it follows that $\mathcal{HS}\subseteq\mathcal{S}$.  Hence, the proof is complete.
\end{proof}

\medskip
\begin{proof}[Proof of Proposition \ref{prop:monotonicity}] The first step will be to show that $\mathcal{L}g_1(x)\geq\mathcal{L}g_2(x)$ on $[0, 1]$.  Letting $F_t(x)=tg_1(x)+(1-t)g_2(x)$, it will suffice to show that $\frac{\partial (\mathcal{L}F_t(x))}{\partial t}\geq 0\ \forall\ x,t\in [0, 1]$.  Using the formula for $\mathcal{L}$ (see Definition \ref{defn:fm1}) along with the fact that $\frac{\partial F_t(x)}{\partial t}=g_1(x)-g_2(x)$, then gives the following expression:$$
\frac{\partial (\mathcal{L}F_t(x))}{\partial t}=3\cdot\frac{x+3}{4}F_t\Big{(}\frac{x+2}{3}\Big{)}^2\Big{(}g_1\Big{(}\frac{x+2}{3}\Big{)}-g_2\Big{(}\frac{x+2}{3}\Big{)}\Big{)}+4\cdot\frac{x+2}{4}F_t\Big{(}\frac{x+1}{3}\Big{)}\Big{(}g_1\Big{(}\frac{x+1}{3}\Big{)}-g_2\Big{(}\frac{x+1}{3}\Big{)}\Big{)}$$
$$-2\cdot\frac{x+2}{4}F_t\Big{(}\frac{x+1}{3}\Big{)}^2\Big{(}g_1\Big{(}\frac{x+2}{3}\Big{)}-g_2\Big{(}\frac{x+2}{3}\Big{)}\Big{)}-4\cdot\frac{x+2}{4}F_t\Big{(}\frac{x+2}{3}\Big{)}F_t\Big{(}\frac{x+1}{3}\Big{)}\Big{(}g_1\Big{(}\frac{x+1}{3}\Big{)}-g_2\Big{(}\frac{x+1}{3}\Big{)}\Big{)}$$
$$-2\cdot\frac{x+1}{4}F_t\Big{(}\frac{x+2}{3}\Big{)}F_t\Big{(}\frac{x}{3}\Big{)}\Big{(}g_1\Big{(}\frac{x+2}{3}\Big{)}-g_2\Big{(}\frac{x+2}{3}\Big{)}\Big{)}-\frac{x+1}{4}F_t\Big{(}\frac{x+2}{3}\Big{)}^2\Big{(}g_1\Big{(}\frac{x}{3}\Big{)}-g_2\Big{(}\frac{x}{3}\Big{)}\Big{)}-2\cdot\frac{x+1}{4}$$
$$F_t\Big{(}\frac{x+1}{3}\Big{)}\Big{(}g_1\Big{(}\frac{x}{3}\Big{)}-g_2\Big{(}\frac{x}{3}\Big{)}\Big{)}-2\cdot\frac{x+1}{4}F_t\Big{(}\frac{x}{3}\Big{)}\Big{(}g_1\Big{(}\frac{x+1}{3}\Big{)}-g_2\Big{(}\frac{x+1}{3}\Big{)}\Big{)}+2\cdot\frac{x+1}{4}F_t\Big{(}\frac{x+2}{3}\Big{)}F_t\Big{(}\frac{x+1}{3}\Big{)}$$
$$\Big{(}g_1\Big{(}\frac{x}{3}\Big{)}-g_2\Big{(}\frac{x}{3}\Big{)}\Big{)}+2\cdot\frac{x+1}{4}F_t\Big{(}\frac{x+2}{3}\Big{)}F_t\Big{(}\frac{x}{3}\Big{)}\Big{(}g_1\Big{(}\frac{x+1}{3}\Big{)}-g_2\Big{(}\frac{x+1}{3}\Big{)}\Big{)}+2\cdot\frac{x+1}{4}F_t\Big{(}\frac{x+1}{3}\Big{)}F_t\Big{(}\frac{x}{3}\Big{)}$$
$$\Big{(}g_1\Big{(}\frac{x+2}{3}\Big{)}-g_2\Big{(}\frac{x+2}{3}\Big{)}\Big{)}+\frac{x+1}{4}\Big{(}g_1\Big{(}\frac{x}{3}\Big{)}-g_2\Big{(}\frac{x}{3}\Big{)}\Big{)}$$

$$=\Big{[}\Big{(}2\cdot\frac{x+3}{4}F_t\Big{(}\frac{x+2}{3}\Big{)}^2-2\cdot\frac{x+2}{4}F_t\Big{(}\frac{x+1}{3}\Big{)}^2\Big{)}+\Big{(}\frac{x+3}{4}F_t\Big{(}\frac{x+2}{3}\Big{)}^2-\frac{x+1}{2}F_t\Big{(}\frac{x+2}{3}\Big{)}F_t\Big{(}\frac{x}{3}\Big{)}\Big{)}+\Big{(}\frac{x+1}{2}$$
$$F_t\Big{(}\frac{x+1}{3}\Big{)}F_t\Big{(}\frac{x}{3}\Big{)}\Big{)}\Big{]}\Big{(}g_1\Big{(}\frac{x+2}{3}\Big{)}-g_2\Big{(}\frac{x+2}{3}\Big{)}\Big{)}+\Big{[}\Big{(}1-F_t\Big{(}\frac{x+2}{3}\Big{)}\Big{)}\Big{(}(x+2)F_t\Big{(}\frac{x+1}{3}\Big{)}-\frac{x+1}{2}F_t\Big{(}\frac{x}{3}\Big{)}\Big{)}\Big{]}$$
$$\Big{(}g_1\Big{(}\frac{x+1}{3}\Big{)}-g_2\Big{(}\frac{x+1}{3}\Big{)}\Big{)}+\Big{[}\frac{x+1}{4}\Big{(}1+F_t\Big{(}\frac{x+2}{3}\Big{)}-2F_t\Big{(}\frac{x+1}{3}\Big{)}\Big{)}\Big{(}1-F_t\Big{(}\frac{x+2}{3}\Big{)}\Big{)}\Big{]}\Big{(}g_1\Big{(}\frac{x}{3}\Big{)}-g_2\Big{(}\frac{x}{3}\Big{)}\Big{)}$$

\medskip
\noindent
Since $F_t$ is a convex combination of the probability generating functions $g_1$ and $g_2$, this means $F_t\in\mathcal{S}$ (for any $t\in [0, 1]$).  It follows that $0\leq F_t\leq 1$ on $[0, 1]$ and that $F_t$ is increasing on $[0, 1]$ (w.r.t. $x$).  This then implies that each of the three terms inside the first set of brackets above is non-negative.  Likewise, it also follows that the expressions inside the second and third sets of brackets are non-negative.  Coupling this with the fact that $g_1\geq g_2$, it can then be concluded that $\frac{\partial (\mathcal{L}F_t(x))}{\partial t}\geq 0\ \forall\ x,t\in [0, 1]$, from which it follows that $\mathcal{L}g_1\geq\mathcal{L}g_2$ on $[0, 1]$.

It is also necessary to establish that $\mathcal{H}g_1\geq\mathcal{H}g_2$ on $[0, 1]$.  Recalling the formula for $\mathcal{H}$ (see Definition \ref{defn:fm2}) and using the fact, established above, that $\mathcal{L}g_1\geq\mathcal{L}g_2$, this task amounts to showing that $\mathcal{G}g_1\geq\mathcal{G}g_2$ (where $\mathcal{G}g(x)=\frac{x+3}{6}g(\frac{x+2}{3})^3+\frac{x+2}{6}(g(\frac{x+1}{3})^2-g(\frac{x+2}{3})g(\frac{x+1}{3})^2)$).  Once again letting $F_t(x)=tg_1(x)+(1-t)g_2(x)$, we find that
\begin{align}
\frac{\partial (\mathcal{G}F_t(x))}{\partial t}&=\frac{x+3}{6}\cdot3F_t\Big{(}\frac{x+2}{3}\Big{)}^2\Big{(}g_1\Big{(}\frac{x+2}{3}\Big{)}-g_2\Big{(}\frac{x+2}{3}\Big{)}\Big{)}+\frac{x+2}{6}\cdot2F_t\Big{(}\frac{x+1}{3}\Big{)}\Big{(}g_1\Big{(}\frac{x+1}{3}\Big{)}-g_2\Big{(}\frac{x+1}{3}\Big{)}\Big{)}\nonumber\\
-\frac{x+2}{6}F_t&\Big{(}\frac{x+1}{3}\Big{)}^2\Big{(}g_1\Big{(}\frac{x+2}{3}\Big{)}-g_2\Big{(}\frac{x+2}{3}\Big{)}\Big{)}-\frac{x+2}{6}\cdot2F_t\Big{(}\frac{x+2}{3}\Big{)}F_t\Big{(}\frac{x+1}{3}\Big{)}\Big{(}g_1\Big{(}\frac{x+1}{3}\Big{)}-g_2\Big{(}\frac{x+1}{3}\Big{)}\Big{)}\nonumber\\
&=\Big{[}3\cdot\frac{x+3}{6}F_t\Big{(}\frac{x+2}{3}\Big{)}^2-\frac{x+2}{6}F_t\Big{(}\frac{x+1}{3}\Big{)}^2\Big{]}\Big{(}g_1\Big{(}\frac{x+2}{3}\Big{)}-g_2\Big{(}\frac{x+2}{3}\Big{)}\Big{)}\nonumber\\
&+\Big{[}2\cdot\frac{x+2}{6}F_t\Big{(}\frac{x+1}{3}\Big{)}-2\cdot\frac{x+2}{6}F_t\Big{(}\frac{x+1}{3}\Big{)}F_t\Big{(}\frac{x+2}{3}\Big{)}\Big{]}\Big{(}g_1\Big{(}\frac{x+1}{3}\Big{)}-g_2\Big{(}\frac{x+1}{3}\Big{)}\Big{)}\nonumber
\end{align}
It then follows from the three facts --(i) $0\leq F_t\leq 1$, (ii) $F_t$ is increasing with respect to $x$, and (iii) $g_1\geq g_2$ -- that both terms in the above sum are non-negative, which means $$\frac{\partial (\mathcal{G}F_t(x))}{\partial t}\geq 0\ \forall\ x,t\in [0, 1]\implies\mathcal{G}g_1\geq\mathcal{G}g_2\implies\mathcal{H}g_1\geq\mathcal{H}g_2$$as desired.

Having established the monotonicity of $\mathcal{L}$ and $\mathcal{H}$ on $\mathcal{S}$, we are now ready to prove the lemma.  To start, define $\tilde{\mathcal{A}}$ to be an operator on $\mathcal{S}\times\mathcal{S}$ where $$\tilde{\mathcal{A}}[f_1,f_2](x)=\frac{1}{3}f_1\Big{(}\frac{x+1}{2}\Big{)}^2+\frac{x}{3}f_1\Big{(}\frac{x}{2}\Big{)}+\frac{x}{3}f_1\Big{(}\frac{x+1}{2}\Big{)}\Big{(}f_1\Big{(}\frac{x+1}{2}\Big{)}-f_1\Big{(}\frac{x}{2}\Big{)}\Big{)}$$
$$+\frac{1}{3}f_2\Big{(}\frac{x}{2}\Big{)}+\frac{1}{3}f_1\Big{(}\frac{x+1}{2}\Big{)}\Big{(}f_2\Big{(}\frac{x+1}{2}\Big{)}-f_2\Big{(}\frac{x}{2}\Big{)}\Big{)}$$Noting that $\mathcal{A}g(x)=\tilde{\mathcal{A}}[\mathcal{L}g,\mathcal{H}g](x)$ and that $\mathcal{LS}\subseteq\mathcal{S}$, $\mathcal{HS}\subseteq\mathcal{S}$, $\mathcal{L}g_1\geq\mathcal{L}g_2$, and $\mathcal{H}g_1\geq\mathcal{H}g_2$, it suffices to show that if $H_1,H_2,G_1,G_2\in\mathcal{S}$ with $H_1\geq G_1$ and $H_2\geq G_2$, then the following inequality holds:
\begin{equation}\label{finmonop}
\tilde{\mathcal{A}}[H_1,H_2](x)\geq\tilde{\mathcal{A}}[G_1,G_2](x)
\end{equation}
Defining $F^{(i)}_t=tH_i+(1-t)G_i$ (for $i=1,2$), if it can be established that
\begin{equation}\label{derivA}
\frac{\partial (\tilde{\mathcal{A}}[F^{(1)}_t,F^{(2)}_t](x))}{\partial t}\geq 0
\end{equation}
$\forall\ t,x\in [0, 1]$, then \eqref{finmonop} will follow.  Writing out the formula for the left side of \eqref{derivA} gives the following expression:$$\frac{2}{3}F^{(1)}_t\Big{(}\frac{x+1}{2}\Big{)}\Big{(}H_1\Big{(}\frac{x+1}{2}\Big{)}-G_1\Big{(}\frac{x+1}{2}\Big{)}\Big{)}+\frac{x}{3}\Big{(}H_1\Big{(}\frac{x}{2}\Big{)}-G_1\Big{(}\frac{x}{2}\Big{)}\Big{)}+\frac{2x}{3}F^{(1)}_t\Big{(}\frac{x+1}{2}\Big{)}\Big{(}H_1\Big{(}\frac{x+1}{2}\Big{)}-G_1\Big{(}\frac{x+1}{2}\Big{)}\Big{)}$$
$$-\frac{x}{3}F^{(1)}_t\Big{(}\frac{x+1}{2}\Big{)}\Big{(}H_1\Big{(}\frac{x}{2}\Big{)}-G_1\Big{(}\frac{x}{2}\Big{)}\Big{)}-\frac{x}{3}F^{(1)}_t\Big{(}\frac{x}{2}\Big{)}\Big{(}H_1\Big{(}\frac{x+1}{2}\Big{)}-G_1\Big{(}\frac{x+1}{2}\Big{)}\Big{)}+\frac{1}{3}\Big{(}H_2\Big{(}\frac{x}{2}\Big{)}-G_2\Big{(}\frac{x}{2}\Big{)}\Big{)}$$
$$+\frac{1}{3}F^{(1)}_t\Big{(}\frac{x+1}{2}\Big{)}\Big{(}H_2\Big{(}\frac{x+1}{2}\Big{)}-G_2\Big{(}\frac{x+1}{2}\Big{)}\Big{)}+\frac{1}{3}F^{(2)}_t\Big{(}\frac{x+1}{2}\Big{)}\Big{(}H_1\Big{(}\frac{x+1}{2}\Big{)}-G_1\Big{(}\frac{x+1}{2}\Big{)}\Big{)}$$
$$-\frac{1}{3}F^{(1)}_t\Big{(}\frac{x+1}{2}\Big{)}\Big{(}H_2\Big{(}\frac{x}{2}\Big{)}-G_2\Big{(}\frac{x}{2}\Big{)}\Big{)}-\frac{1}{3}F^{(2)}_t\Big{(}\frac{x}{2}\Big{)}\Big{(}H_1\Big{(}\frac{x+1}{2}\Big{)}-G_1\Big{(}\frac{x+1}{2}\Big{)}\Big{)}$$

$$=\Big{[}\Big{(}2\cdot\frac{x+1}{3}F^{(1)}_t\Big{(}\frac{x+1}{2}\Big{)}-\frac{x}{3}F^{(1)}_t\Big{(}\frac{x}{2}\Big{)}\Big{)}+\Big{(}\frac{1}{3}F^{(2)}_t\Big{(}\frac{x+1}{2}\Big{)}-\frac{1}{3}F^{(2)}_t\Big{(}\frac{x}{2}\Big{)}\Big{)}\Big{]}\Big{(}H_1\Big{(}\frac{x+1}{2}\Big{)}-G_1\Big{(}\frac{x+1}{2}\Big{)}\Big{)}$$
$$+\Big{[}\frac{x}{3}-\frac{x}{3}F^{(1)}_t\Big{(}\frac{x+1}{2}\Big{)}\Big{]}\Big{(}H_1\Big{(}\frac{x}{2}\Big{)}-G_1\Big{(}\frac{x}{2}\Big{)}\Big{)}+\Big{[}\frac{1}{3}F^{(1)}_t\Big{(}\frac{x+1}{2}\Big{)}\Big{]}\Big{(}H_2\Big{(}\frac{x+1}{2}\Big{)}-G_2\Big{(}\frac{x+1}{2}\Big{)}\Big{)}$$
$$+\Big{[}\frac{1}{3}-\frac{1}{3}F^{(1)}_t\Big{(}\frac{x+1}{2}\Big{)}\Big{]}\Big{(}H_2\Big{(}\frac{x}{2}\Big{)}-G_2\Big{(}\frac{x}{2}\Big{)}\Big{)}$$
Now noting that $F^{(1)}_t,F^{(2)}_t\in\mathcal{S}$ (implying they are increasing and between $0$ and $1$), and recalling that $H_i\geq G_i$ for $i=1,2$, we see that \eqref{derivA} follows.  This then implies \eqref{finmonop}, which implies $\mathcal{A}g_1\geq\mathcal{A}g_2$.  Hence, the proof of the proposition is complete.
\end{proof}

\medskip
\noindent
2.5. {\bf Completing the proof of Theorem \ref{theorem:neat1}.} Having established that $\mathcal{A}$ is monotone, it follows that $\mathcal{A}^nf\leq\mathcal{A}^n1\ \forall\ n\geq 1$.  Hence, to show that the expression on the left goes to $0$, it suffices to show that $\mathcal{A}^n1\rightarrow 0$ on $[0, 1)$.  This will be achieved by employing a method referred to in \cite{HJJ1} as Poisson thinning.  Specifically, it involves establishing the existence of a sequence $0=a_0<a_1<a_2<\dots$ (diverging to infinity) such that $\mathcal{A}^n1\leq e^{a_n(x-1)}$ (the probability generating function for $\text{Poiss}(a_n)$) for all $n\geq 0$.  The existence of this sequence is established in two parts.  First, in Proposition 2.12 it is shown that $\forall\ a\geq 15$, $\mathcal{A}[e^{a(x-1)}]\leq e^{(a+\epsilon)(x-1)}$ on $[0, 1]$ (where $\epsilon=\frac{1}{20}$).  It then follows from a simple induction argument which relies on the monotonicity of $\mathcal{A}$ established in Proposition \ref{prop:monotonicity}, that $\mathcal{A}^n[e^{a(x-1)}]\leq e^{(a+n\epsilon)(x-1)}\ \forall\ n\geq 1$.  From this point, establishing the existence of the sequence $\left\{a_n\right\}$ reduces to establishing the existence of a finite sequence $0=a_0<a_1<\dots <a_N$ (where $a_N\geq 15$) such that $\mathcal{A}^n1\leq e^{a_n(x-1)}$ on $[0, 1]\ \forall\ n$ with $0\leq n\leq N$.  This is accomplished (with the help of a Python program) in Proposition 2.14, where we inductively construct a sequence $0=a_0<a_1<\dots<a_N$ satisfying the above constraints.  Along with Proposition 2.12, this will then establish the existence of $\left\{a_n\right\}$.  The result $\mathcal{A}^n 1\rightarrow 0$ on $[0, 1)$ follows immediately, which then implies $\mathcal{A}^n f\rightarrow 0$ on $[0, 1)$.  As explained at the beginning of the previous section, this is then sufficient for establishing Theorem \ref{theorem:neat1}.

\begin{prop}\label{prop:uppermonot}
If $a\geq 15$ then $\mathcal{A}[e^{a(x-1)}]\leq e^{(a+\frac{1}{20})(x-1)}$ on $[0, 1]$.
\end{prop}

\begin{proof} The first step will be to define a simple expression $\Psi (x,a)$ to serve as an upper bound on $\mathcal{A}[e^{a(x-1)}]$ (for $a\geq 15$).  To start, note that$$\mathcal{A}[g](x)=\frac{x}{3}\mathcal{L}[g]\Big(\frac{x}{2}\Big)+\frac{x+1}{3}\Big(\mathcal{L}[g]\Big(\frac{x+1}{2}\Big)\Big)^2-\frac{x}{3}\mathcal{L}[g]\Big(\frac{x+1}{2}\Big)\mathcal{L}[g]\Big(\frac{x}{2}\Big)$$
$$+\frac{1}{3}\mathcal{H}[g]\Big(\frac{x}{2}\Big)+\frac{1}{3}\mathcal{L}[g]\Big(\frac{x+1}{2}\Big)\mathcal{H}[g]\Big(\frac{x+1}{2}\Big)-\frac{1}{3}\mathcal{L}[g]\Big(\frac{x+1}{2}\Big)\mathcal{H}[g]\Big(\frac{x}{2}\Big)$$

\begin{equation}\label{Abound1}
\leq \frac{x}{3}\mathcal{L}[g]\Big(\frac{x}{2}\Big)+\frac{x+1}{3}\Big(\mathcal{L}[g]\Big(\frac{x+1}{2}\Big)\Big)^2+\frac{1}{3}\mathcal{H}[g]\Big(\frac{x}{2}\Big)+\frac{1}{3}\mathcal{L}[g]\Big(\frac{x+1}{2}\Big)\mathcal{H}[g]\Big(\frac{x+1}{2}\Big)
\end{equation}

\medskip
\noindent
$\forall\ g\in\mathcal{S}$.  To bound \eqref{Abound1} above (for $g(x)=e^{a(x-1)}$) we'll first obtain upper bounds for $\mathcal{L}[e^{a(x-1)}]$ and $\mathcal{H}[e^{a(x-1)}]$ as follows:$$\mathcal{L}[e^{a(x-1)}]=\frac{x+3}{4}e^{a(x-1)}+2\cdot\frac{x+2}{4}\Big(e^{\frac{2a}{3}(x-2)}-e^{a(x-\frac{5}{3})}\Big)+\frac{x+1}{4}\Big(e^{\frac{a}{3}(x-3)}-2e^{\frac{2a}{3}(x-\frac{5}{2})}-e^{a(x-\frac{5}{3})}+2e^{a(x-2)}\Big)$$
Observing that for all $x\in [0, 1]$, $2\cdot\frac{x+2}{4}e^{a(x-\frac{5}{3})}\geq e^{-\frac{a}{3}}e^{\frac{2a}{3}(x-2)}$, $2\cdot\frac{x+1}{4}e^{\frac{2a}{3}(x-\frac{5}{2})}\geq\frac{1}{2}e^{-\frac{a}{3}}e^{\frac{2a}{3}(x-2)}$, and $\frac{x+1}{4}e^{a(x-\frac{5}{3})}\geq\frac{1}{4}e^{-\frac{a}{3}}e^{\frac{2a}{3}(x-2)}$, along with the fact that $2\cdot\frac{x+2}{4}e^{\frac{2a}{3}(x-2)}\leq\frac{3}{2}e^{\frac{2a}{3}(x-1)}$ and $2\cdot\frac{x+1}{4}e^{a(x-2)}\leq e^{-\frac{a}{3}}e^{\frac{2a}{3}(x-2)}$, we find that if we make the given substitutions in the expression for $\mathcal{L}[e^{a(x-1)}]$ above, it gives$$\mathcal{L}[e^{a(x-1)}]\leq\frac{x+3}{4}e^{a(x-1)}+\frac{x+1}{4}e^{\frac{a}{3}(x-3)}+ce^{\frac{2a}{3}(x-2)}$$(where $c=\frac{3}{2}-\frac{3}{4}e^{-\frac{a}{3}}$).  The above upper bound on $\mathcal{L}[e^{a(x-1)}]$ will be denoted as $l_a(x)$.  Now noting that$$\mathcal{H}[e^{a(x-1)}]=\frac{x+3}{4}e^{a(x-1)}+2\cdot\frac{x+2}{6}\Big(e^{\frac{2a}{3}(x-2)}-e^{a(x-\frac{5}{3})}\Big)+\frac{x+1}{12}\Big(e^{\frac{a}{3}(x-3)}-2e^{\frac{2a}{3}(x-\frac{5}{2})}-e^{a(x-\frac{5}{3})}+2e^{a(x-2)}\Big)$$applying a similar set of inequalities then gives the bound$$\mathcal{H}[e^{a(x-1)}]\leq\frac{x+3}{4}e^{a(x-1)}+\frac{x+1}{12}e^{\frac{a}{3}(x-3)}+de^{\frac{2a}{3}(x-2)}$$(where $d=1-\frac{7}{12}e^{-\frac{a}{3}}$).  This upper bound on $\mathcal{H}[e^{a(x-1)}]$ will be denoted as $h_a(x)$.

Combining the above bounds with \eqref{Abound1} we obtain the inequality$$\mathcal{A}[e^{a(x-1)}]\leq \frac{x}{3}l_a\Big(\frac{x}{2}\Big)+\frac{x+1}{3}l_a\Big(\frac{x+1}{2}\Big)^2+\frac{1}{3}h_a\Big(\frac{x}{2}\Big)+\frac{1}{3}l_a\Big(\frac{x+1}{2}\Big)h_a\Big(\frac{x+1}{2}\Big)$$Writing out this full expression gives the following:$$\mathcal{A}[e^{a(x-1)}]\leq \frac{x}{3}\Big(\frac{x+6}{8}e^{\frac{a}{2}(x-2)}+\frac{x+2}{8}e^{\frac{a}{6}(x-6)}+ce^{\frac{a}{3}(x-4)}\Big)+\frac{x+1}{3}\Big(\Big(\frac{x+7}{8}\Big)^2 e^{a(x-1)}+\Big(\frac{x+3}{8}\Big)^2e^{\frac{a}{3}(x-5)}$$
$$+c^2 e^{\frac{2a}{3}(x-3)}+2\cdot\frac{x+7}{8}\cdot\frac{x+3}{8}e^{\frac{2a}{3}(x-2)}+2\cdot\frac{x+7}{8}\cdot ce^{\frac{5a}{6}(x-\frac{9}{5})}+2\cdot\frac{x+3}{8}\cdot ce^{\frac{a}{2}(x-\frac{11}{3})}\Big)$$
$$+\frac{1}{3}\Big(\frac{x+6}{8}e^{\frac{a}{2}(x-2)}+\frac{x+2}{24}e^{\frac{a}{6}(x-6)}+de^{\frac{a}{3}(x-4)}\Big)+\frac{1}{3}\Big(\Big(\frac{x+7}{8}\Big)^2e^{a(x-1)}+\frac{x+3}{8}\cdot\frac{x+3}{24}e^{\frac{a}{3}(x-5)}$$
$$+cde^{\frac{2a}{3}(x-3)}+\frac{4}{3}\cdot\frac{x+7}{8}\cdot\frac{x+3}{8}e^{\frac{2a}{3}(x-2)}+(c+d)\frac{x+7}{8}e^{\frac{5a}{6}(x-\frac{9}{5})}+(\frac{c}{3}+d)\frac{x+3}{8}e^{\frac{a}{2}(x-\frac{11}{3})}\Big)$$

$$=\frac{x+2}{3}\Big(\frac{x+7}{8}\Big)^2 e^{a(x-1)}+\frac{x+1}{3}\cdot\frac{x+6}{8}e^{\frac{a}{2}(x-2)}+\frac{x+\frac{1}{3}}{3}\cdot\frac{x+2}{8}e^{\frac{a}{6}(x-6)}+e^{\frac{2a}{3}(x-2)}\Big(\frac{x}{3}\cdot ce^{-\frac{a}{3}x}+\frac{x+1}{3}\cdot\Big(\frac{x+3}{8}\Big)^2 $$
$$e^{-\frac{a}{3}(x+1)}+\frac{x+1}{3}\cdot c^2e^{-\frac{2a}{3}}+2\cdot\frac{x+1}{3}\cdot\frac{x+7}{8}\cdot\frac{x+3}{8}+2\cdot\frac{x+1}{3}\cdot\frac{x+7}{8}\cdot ce^{\frac{a}{6}(x-1)}+2\cdot\frac{x+1}{3}\cdot\frac{x+3}{8}\cdot ce^{-\frac{a}{6}(x+3)}$$
$$+\frac{d}{3}e^{-\frac{a}{3}x}+\Big(\frac{x+3}{24}\Big)^2e^{-\frac{a}{3}(x+1)}+\frac{c}{3}de^{-\frac{2a}{3}}+\frac{4}{9}\cdot\frac{x+7}{8}\cdot\frac{x+3}{8}+(c+d)\frac{x+7}{24}e^{\frac{a}{6}(x-1)}+\frac{1}{3}\Big(\frac{c}{3}+d\Big)\frac{x+3}{8}e^{-\frac{a}{6}(x+3)}\Big)$$

\medskip
\noindent
An upper bound for the long expression in parentheses above can be obtained by replacing $x$ with $1$ wherever it is part of an increasing expression (such as $\frac{x}{3}$ or $e^{ax}$) and replacing it with $0$ wherever it is part of a decreasing expression.  After simplifying, this gives the following inequality:$$\mathcal{A}[e^{a(x-1)}]\leq\frac{x+2}{3}\Big(\frac{x+7}{8}\Big)^2e^{a(x-1)}+\frac{x+1}{3}\cdot\frac{x+6}{8}e^{\frac{a}{2}(x-2)}+\frac{x+\frac{1}{3}}{3}\cdot\frac{x+2}{8}e^{\frac{a}{6}(x-6)}$$
$$+\Big(\frac{41}{9}-\frac{61}{36}e^{-\frac{a}{3}}+\frac{5}{4}e^{-\frac{a}{2}}+2e^{-\frac{2a}{3}}-\frac{23}{36}e^{-\frac{5a}{6}}-\frac{49}{24}e^{-a}+\frac{25}{48}e^{-\frac{4a}{3}}\Big)e^{\frac{2a}{3}(x-2)}$$

\medskip
\noindent
Note that for $a\geq 3$ the following string of inequalities holds$$\frac{41}{9}-\frac{61}{36}e^{-\frac{a}{3}}+\frac{5}{4}e^{-\frac{a}{2}}+2e^{-\frac{2a}{3}}-\frac{23}{36}e^{-\frac{5a}{6}}-\frac{49}{24}e^{-a}+\frac{25}{48}e^{-\frac{4a}{3}}\leq\frac{41}{9}-\frac{61}{36}e^{-\frac{a}{3}}+e^{-\frac{a}{3}}\Big(\frac{5}{4}e^{-\frac{a}{6}}+2e^{-\frac{a}{3}}\Big)\leq\frac{41}{9}$$Hence, we now finally define $\Psi(x,a)$ to be$$\Psi(x,a)=\frac{x+2}{3}\Big(\frac{x+7}{8}\Big)^2e^{a(x-1)}+\frac{x+1}{3}\cdot\frac{x+6}{8}e^{\frac{a}{2}(x-2)}+\frac{x+\frac{1}{3}}{3}\cdot\frac{x+2}{8}e^{\frac{a}{6}(x-6)}+\frac{41}{9}e^{\frac{2a}{3}(x-2)}$$From the above computations, it follows that $\mathcal{A}[e^{a(x-1)}]\leq\Psi(x,a)$ on $[0,1]$ for $a\geq 15$ as desired (though as we saw above, having $a\geq 3$ is sufficient for this inequality to hold).

Now that $\Psi(x, a)$ has been defined, we'll proceed to prove the proposition by splitting up the interval $[0, 1]$ into four parts, and showing that the inequality stated in the proposition holds for all $x$ in each one of them.

\medskip
\noindent
(i) $x\in [1-c(a), 1]$ (where $c(a)=a^{-\frac{9}{4}}$).

\medskip
\noindent
Since $\mathcal{A}[e^{a(x-1)}]$ is a convex function of $x$ (this follows from it being a probability generating function), this means that for any $c\in [0,1]$ we have $\mathcal{A}[e^{a(x-1)}]\leq\mathcal{A}[e^{a(c-1)}]+\Big(1-\mathcal{A}[e^{a(c-1)}]\Big)\Big(\frac{x-c}{1-c}\Big)\ \forall\ x\in [c,1]$.  Using the fact that $\mathcal{A}[e^{a(x-1)}]\leq\Psi(x,a)$ (for $a\geq 15$), it follows that $\mathcal{A}[e^{a(x-1)}]\leq\Psi(c,a)+\Big(1-\Psi(c,a)\Big)\Big(\frac{x-c}{1-c}\Big)$ on $[c,1]$.  Noting that $e^{(a+\frac{1}{20})(x-1)}$ is itself a convex function of $x$ that has derivative $a+\frac{1}{20}$ at $x=1$, it follows that $e^{(a+\frac{1}{20})(x-1)}\geq 1-(a+\frac{1}{20})(1-x)$ on $[0,1]$.  Putting these last two observations together, we find that if we can establish that
\begin{equation}\label{psbound}
\Psi(1-c(a),a)\leq 1-(a+\frac{1}{20})(1-(1-c(a)))
\end{equation}
then it will follow that$$\mathcal{A}[e^{a(x-1)}]\leq 1-\Big(a+\frac{1}{20}\Big)\Big(1-(1-c(a))\Big)+\Big(a+\frac{1}{20}\Big)\Big(1-(1-c(a))\Big)\Big(\frac{x-(1-c(a))}{1-(1-c(a))}\Big)$$
$$=1-\Big(a+\frac{1}{20}\Big)\Big(1-x\Big)\leq e^{(a+\frac{1}{20})(x-1)}$$
for all $x\in [1-c(a),1]$.

Now using the formula for $\Psi$, we get the string of inequalities$$\Psi(1-c(a),a)\leq\Big(1-\frac{c(a)}{3}\Big)\Big(1-\frac{c(a)}{8}\Big)^2 e^{-ac(a)}+\frac{7}{12}e^{-\frac{a}{2}}+\frac{1}{6}e^{-\frac{5a}{6}}+\frac{41}{9}e^{-\frac{2a}{3}}$$
$$\leq e^{-(a+\frac{7}{12})c(a)}+\frac{7}{12}e^{-\frac{a}{2}}+\frac{85}{18}e^{-\frac{2a}{3}}\leq 1-\Big(a+\frac{7}{12}\Big)c(a)+\frac{13}{24}a^2 c(a)^2+\frac{7}{12}e^{-\frac{a}{2}}+\frac{85}{18}e^{-\frac{2a}{3}}$$(where the last inequality follows from the fact that $e^{-x}\leq 1-x+\frac{x^2}{2}$ for $x\in [0,1]$, and the fact that $\Big(a+\frac{7}{12}\Big)^2\leq\frac{13}{12}a^2$ for $a\geq 15$).  Plugging $c(a)=a^{-\frac{9}{4}}$ into the above expression then gives$$\Psi(1-c(a),a)\leq 1-\Big(a+\frac{7}{12}\Big)c(a)+\Big(\frac{13}{24}a^{-\frac{1}{4}}+\frac{7}{12}a^{\frac{9}{4}}e^{-\frac{a}{2}}+\frac{85}{18}a^{\frac{9}{4}}e^{-\frac{2a}{3}}\Big)c(a)$$Now to establish \eqref{psbound} it just needs to be shown that
\begin{equation}\label{bound2}
\frac{13}{24}a^{-\frac{1}{4}}+\frac{7}{12}a^{\frac{9}{4}}e^{-\frac{a}{2}}+\frac{85}{18}a^{\frac{9}{4}}e^{-\frac{2a}{3}}\leq\frac{7}{12}-\frac{1}{20}
\end{equation}
for $a\geq 15$.  So observe the string of inequalities below (which holds for $a\geq\frac{9}{2}$), where the left side is equal to the derivative of the left side of \eqref{bound2}.$$-\frac{13}{96}a^{-\frac{5}{4}}+\frac{9}{4}a^{\frac{5}{4}}\Big(\frac{7}{12}e^{-\frac{a}{2}}+\frac{85}{18}e^{-\frac{2a}{3}}\Big)-a^{\frac{9}{4}}\Big(\frac{1}{2}\cdot\frac{7}{12}e^{-\frac{a}{2}}+\frac{2}{3}\cdot\frac{85}{18}e^{-\frac{2a}{3}}\Big)<\Big(\frac{9}{4}a^{\frac{5}{4}}-\frac{1}{2}a^{\frac{9}{4}}\Big)\Big(\frac{7}{12}e^{-\frac{a}{2}}+\frac{85}{18}e^{-\frac{2a}{3}}\Big)<0$$Combining this with the fact that the left side of \eqref{bound2} equals $.513<\frac{7}{12}-\frac{1}{20}$ at $a=15$, we find that \eqref{bound2} does indeed hold for $a\geq 15$ which, as was shown, implies that $\mathcal{A}[e^{a(x-1)}]\leq e^{(a+\frac{1}{20})(x-1)}$ on $[1-c(a),1]$.

\medskip
\noindent
(ii) $x\in [\frac{1}{2},1-c(a))$.

\medskip
\noindent
Denoting $e^{-a(x-1)}\Psi(x,a)$ as $Q(x,a)$ (for $a\geq 15$), it suffices to show that $Q(x,a)\leq e^{\frac{1}{20}(x-1)}$ on $[\frac{1}{2},1-c(a))$.  Since we saw in (i) that $\Psi(1-c(a),a)\leq 1-\Big(a+\frac{1}{20}\Big)c(a)\leq e^{(a+\frac{1}{20})((1-c(a))-1)}$, it follows that $Q(1-c(a),a)\leq e^{\frac{1}{20}((1-c(a))-1)}$, which implies that to prove $Q(x,a)\leq e^{\frac{1}{20}(x-1)}$, it suffices to prove that the right side of$$\frac{\partial \Big(e^{\frac{1}{20}(x-1)}\Big)}{\partial x}\leq\frac{1}{20}\leq\frac{\partial Q(x,a)}{\partial x}$$holds on $[\frac{1}{2},1-c(a))$.  Computing the formula for the expression on the right, we get$$\frac{\partial Q(x,a)}{\partial x}=\frac{1}{3}\Big(\frac{x+7}{8}\Big)^2+\frac{1}{4}\cdot\frac{x+2}{3}\cdot\frac{x+7}{8}+\frac{1}{3}\cdot\frac{x+6}{8}e^{-\frac{a}{2}x}+\frac{1}{8}\cdot\frac{x+1}{3}e^{-\frac{a}{2}x}-\frac{a}{2}\cdot\frac{x+1}{3}\cdot\frac{x+6}{8}e^{-\frac{a}{2}x}$$
$$+\frac{1}{3}\cdot\frac{x+2}{8}e^{-\frac{5a}{6}x}+\frac{1}{8}\cdot\frac{x+\frac{1}{3}}{3}e^{-\frac{5a}{6}x}-\frac{5a}{6}\cdot\frac{x+\frac{1}{3}}{3}\cdot\frac{x+2}{8}e^{-\frac{5a}{6}x}-\frac{a}{3}\cdot\frac{41}{9}e^{-\frac{a}{3}(x+1)}$$
$$\geq\frac{1}{3}\Big(\frac{x+7}{8}\Big)^2+\frac{1}{4}\cdot\frac{x+2}{3}\cdot\frac{x+7}{8}-\frac{a}{2}\cdot\frac{x+1}{3}\cdot\frac{x+6}{8}e^{-\frac{a}{2}x}-\frac{5a}{6}\cdot\frac{x+\frac{1}{3}}{3}\cdot\frac{x+2}{8}e^{-\frac{5a}{6}x}-\frac{a}{3}\cdot\frac{41}{9}e^{-\frac{a}{3}(x+1)}$$Plugging in $x=\frac{1}{2}$ for the exponential functions and the polynomial expressions that follow a $'+'$, and $x=1$ for the polynomial expressions that follow a $'-'$, we find that the expression on the right side of the inequality is greater than or equal to$$\frac{1}{3}\Big(\frac{15}{16}\Big)^2+\frac{1}{4}\cdot\frac{5}{6}\cdot\frac{15}{16}-\frac{a}{2}\cdot\frac{2}{3}\cdot\frac{7}{8}e^{-\frac{a}{4}}-\frac{5a}{6}\cdot\frac{4}{9}\cdot\frac{3}{8}e^{-\frac{5a}{12}}-\frac{a}{3}\cdot\frac{41}{9}e^{-\frac{a}{2}}$$on $[\frac{1}{2},1-c(a))$.  Simplifying, and using the string of inequalities above, gives
\begin{equation}\label{derivqb}\frac{\partial Q(x,a)}{\partial x}\geq \frac{125}{256}-\frac{7a}{24}e^{-\frac{a}{4}}-\frac{5a}{36}e^{-\frac{5a}{12}}-\frac{41a}{27}e^{-\frac{a}{2}}\end{equation}
on this interval.  If we differentiate this expression with respect to $a$ we get$$\Big(\frac{a}{4}-1\Big)\cdot\frac{7}{24}e^{-\frac{a}{4}}+\Big(\frac{5a}{12}-1\Big)\cdot\frac{5}{36}e^{-\frac{5a}{12}}+\Big(\frac{a}{2}-1\Big)\cdot\frac{41}{27}e^{-\frac{a}{2}}\geq 0$$(recall we're assuming $a\geq 15$).  Coupling this with the fact that the expression on the right side of \eqref{derivqb}, when evaluated at $a=15$, is equal to $.369>\frac{1}{20}$, we indeed find that $\frac{\partial Q(x,a)}{\partial x}\geq\frac{1}{20}$ on $[\frac{1}{2},1-c(a))$ for $a\geq 15$.  As was shown, this implies that $Q(x,a)\leq e^{\frac{1}{20}(x-1)}$, which implies $\mathcal{A}[e^{a(x-1)}]\leq e^{(a+\frac{1}{20})(x-1)}$ on $[\frac{1}{2},1-c(a))$ for $a\geq 15$ as desired.

\medskip
\noindent
(iii) $x\in[\frac{1}{8}, \frac{1}{2})$.

\medskip
\noindent
Once again it suffices to show that $Q(x,a)\leq e^{\frac{1}{20}(x-1)}$ (this time on $[\frac{1}{8},\frac{1}{2})$).  Taking the formula for $Q(x,a)=e^{-a(x-1)}\Psi(x,a)$ and substituting $\frac{1}{2}$ for $x$ when it is part of a polynomial function, and $\frac{1}{8}$ when it is part of an exponential expression (with negative exponent), we find that$$Q(x,a)\leq\frac{375}{512}+\frac{13}{32}e^{-\frac{a}{16}}+\frac{25}{288}e^{-\frac{5a}{48}}+\frac{41}{9}e^{-\frac{3a}{8}}$$ for $x\in[\frac{1}{8},\frac{1}{2})$.  Since the expression on the right is a decreasing function of $a$, plugging in $a=15$ shows that$$Q(x,a)\leq\frac{375}{512}+\frac{13}{32}e^{-\frac{15}{16}}+\frac{25}{288}e^{-\frac{25}{16}}+\frac{41}{9}e^{-\frac{45}{8}}\approx .926<e^{\frac{1}{20}(\frac{1}{8}-1)}\leq e^{\frac{1}{20}(x-1)}$$on $[\frac{1}{8},\frac{1}{2})$ for $a\geq 15$, thus giving the desired inequality.

\medskip
\noindent
(iv) $x\in[0,\frac{1}{8})$.

\medskip
\noindent
Using the exact same method that was used in (iii), but plugging in $0$ and $\frac{1}{8}$ in place of $\frac{1}{8}$ and $\frac{1}{2}$ respectively, we find that$$Q(x,a)\leq\frac{17}{24}\Big(\frac{57}{64}\Big)^2+\frac{3}{8}\cdot\frac{49}{64}+\frac{11}{72}\cdot\frac{17}{64}+\frac{41}{9}e^{-5}\approx .9203<e^{-\frac{1}{20}}\leq e^{\frac{1}{20}(x-1)}$$on $[0,\frac{1}{8})$ for $a\geq 15$, once again yielding the desired inequality.

\medskip
\noindent
Combining parts (i)-(iv) we find that $\mathcal{A}[e^{a(x-1)}]\leq e^{(a+\frac{1}{20})(x-1)}$ does hold on $[0,1]$ for $a\geq 15$, thus completing the proof of the proposition.
\end{proof}

\medskip
\begin{cor}\label{cor:Aboundind} If $a\geq 15$ and $n\geq 1$ then $\mathcal{A}^n[e^{a(x-1)}]\leq e^{(a+n\epsilon)(x-1)}$ (where $\epsilon=\frac{1}{20}$).
\end{cor}

\begin{proof}
We know from the previous result that the statement holds for $n=1$.  Now assume it holds for some $n\geq 1$.  Then by the monotonicity of $\mathcal{A}$ on $\mathcal{S}$ (established in Proposition \ref{prop:monotonicity}), along with Proposition \ref{prop:uppermonot}, it follows that$$\mathcal{A}^{n+1}[e^{a(x-1)}]=\mathcal{A}\Big[\mathcal{A}^n[e^{a(x-1)}]\Big]\leq\mathcal{A}[e^{(a+n\epsilon)(x-1)}]\leq e^{(a+(n+1)\epsilon)(x-1)}$$on $[0,1]$.  By induction we then find that $\mathcal{A}^n[e^{a(x-1)}]\leq e^{(a+n\epsilon)(x-1)}$ on $[0,1]$ for all $n\geq 1$.
\end{proof}

Having proven Proposition \ref{prop:uppermonot} and it's corollary, our last significant task is to establish the following result.

\begin{prop}\label{prop:lowermonot} There exists a finite sequence $0=a_0<a_1<\dots<a_N$ (with $a_N\geq 15$) such that $\mathcal{A}^n 1\leq e^{a_n(x-1)}$ on $[0,1]$ for all $n$ with $0\leq n\leq N$.
\end{prop}

\noindent
The proof of Proposition \ref{prop:lowermonot} will make use of the following lemma.

\begin{lemma}\label{lemma:lemconv} Let $f_1$ and $f_2$ be convex increasing functions on $[0,1]$ where $f_1$ is differentiable and $f_1(1)=f_2(1)$.  Suppose there is a finite sequence $1=c_0>c_1>\dots>c_n=0$ that satisfies
\begin{equation}\label{convrel}
f_2(c_{j+1})\leq f_1(c_j)-(c_j-c_{j+1})f'_1(c_j)
\end{equation}
for all $j$ with $0\leq j<n$.  Then $f_1(x)\geq f_2(x)\ \forall\ x\in[0,1]$.
\end{lemma}

\begin{proof}
Assume $f_1(c_j)\geq f_2(c_j)$ for some $j<n$.  We know by the convexity (and differentiability) of $f_1$ that $f_1(t)\geq f_1(c_j)-f'_1(c_j)(c_j-t)$ for $t\in[c_{j+1},c_j]$.  By the convexity of $f_2$ it follows that$$f_2(t)\leq f_2(c_j)-\frac{f_2(c_j)-f_2(c_{j+1})}{c_j-c_{j+1}}(c_j-t)\leq f_1(c_j)-f'_1(c_j)(c_j-t)\leq f_1(t)$$for $t\in[c_{j+1},c_j]$ (where the middle inequality follows from $f_1(c_j)\geq f_2(c_j)$, \eqref{convrel}, and the fact that both functions are linear).  Since $f_1(1)\geq f_2(1)$, it follows by induction that $f_1(t)\geq f_2(t)\ \forall\ t\in[0,1]$.
\end{proof}

\begin{proof}[Proof of Proposition \ref{prop:lowermonot}] Let $u\geq 0$, $a>0$, and $c_i=\frac{256-i}{256}$ for $0\leq i\leq 256$.  Recalling that $\mathcal{A}[e^{u(x-1)}]$ is a probability generating function (implying it is increasing and convex on $[0,1]$) and noting that $e^{(u+a)(x-1)}$ is increasing, convex, and differentiable on $[0,1]$, along with the fact that the two functions both equal $1$ at $x=1$, we find that if \eqref{convrel} holds for each $i$ with $0\leq i<256$ (where $f_1(x)=e^{(u+a)(x-1)}$ and $f_2(x)=\mathcal{A}[e^{u(x-1)}]$), then it will follow from Lemma \ref{lemma:lemconv} that $\mathcal{A}[e^{u(x-1)}]\leq e^{(u+a)(x-1)}$ on $[0,1]$.  Now observe the attached Python program.  For each pass through the while loop (see line 45) it checks to see if \eqref{convrel} holds (at each $c_i$) for $a=\frac{1}{16}$, $f_1(x)=e^{(u+a)(x-1)}$, and $f_2(x)=\mathcal{A}[e^{u(x-1)}]$.  If \eqref{convrel} does hold at each $c_i$ then $u$ is increased by $\frac{1}{16}$ and we repeat the process with the new values of $u$, $f_1$, and $f_2$.  If not, $a$ is set to $\frac{1}{32}$ and it tests to see if \eqref{convrel} holds for each $i$ for this value of $a$.  If so, $u$ is increased by $\frac{1}{32}$ and the process is repeated for the new $u$, $f_1$, and $f_2$ (again starting with $a=\frac{1}{16}$).  If not, it tests again with $a=\frac{3}{256}$.  If \eqref{convrel} holds at each $c_i$ then the process repeats with $u$, $f_1$, and $f_2$ adjusted accordingly.  If not, then the while loop terminates.  The loop keeps running until either it terminates (as described above) because \eqref{convrel} fails to hold at some $c_i$ for $a$ equal to each of the three specified values ($\frac{1}{16}$, $\frac{1}{32}$, and $\frac{3}{256}$), or because $m=341$ (i.e. we've passed through the loop $340$ times).  In order to ensure that the program does not return a false negative (as a result of rounding) when evaluating the inequality on line 50, interval arithmetic is employed (see $\text{https://en.wikipedia.org/wiki/Interval\_arithmetic}$ for a definition) so that, for each $a, u, i$ combination that is considered, the loop only fails to break if $A$ (an interval containing the precise value of $f_1(c_j)-(c_j-c_{j+1})f'_1(c_j)$) lies entirely to the right of $B$ (an interval containing the precise value of $f_2(c_{j+1})$). At the end, the program prints the final values of $m$ and $u$.  Upon running the program you will find that these values are $341$ and $15.203125$ respectively (the program prints the current value of $m$ as it runs, and should take about eight minutes to finish).

Now for $0\leq n\leq 340$ let $a_n$ represent the value taken by $u$ following the nth pass through the loop.  Hence, $0=a_0<a_1<\dots<a_{340}=15.203125$ and $a_{j+1}-a_j\in\left\{\frac{1}{16},\frac{1}{32},\frac{3}{256}\right\}$ for each $0\leq j<340$.  Furthermore, since the program output indicates that $340$ passes through the loop were completed, this implies that \eqref{convrel} holds (at each $c_i$ for $0\leq i<256$) for each $0\leq j\leq 340$ (where $f_1(x)=e^{a_{j+1}(x-1)}$ and $f_2(x)=\mathcal{A}[e^{a_j(x-1)}]$).  By Lemma \ref{lemma:lemconv}, this implies that $\mathcal{A}[e^{a_j(x-1)}]\leq e^{a_{j+1}(x-1)}$ on $[0,1]$ for every $0\leq j<340$.  It then follows from the same induction argument that was used to prove Corollary \ref{cor:Aboundind} that $\mathcal{A}^n 1\leq e^{a_n(x-1)}$ for every $n$ with $0\leq n\leq 340$.  Hence, we find that the $a_n$ terms satisfy the conditions given in the statement of the proposition.  Hence, the proof is complete.
\end{proof}

With Proposition \ref{prop:lowermonot} established, the proof of Theorem \ref{theorem:neat1} can now be completed.

\begin{proof}[Proof of Theorem \ref{theorem:neat1}] Proposition \ref{prop:lowermonot} and Corollary \ref{cor:Aboundind} together indicate that $\mathcal{A}^n 1\rightarrow 0$ on $[0,1)$ as $n\rightarrow\infty$.  Since the monotonicity of $\mathcal{A}$ implies that $\mathcal{A}^n f\leq\mathcal{A}^n 1\ \forall\ n\geq 0$, it follows that $\mathcal{A}^n f\rightarrow 0$ on $[0,1)$ as $n\rightarrow\infty$.  Since $f$ is known to be a fixed point of $\mathcal{A}$, this then means that $f(x)=0$.  As explained in the introduction, this implies that $\mathbb{P}(V=\infty)=1$.  Recalling from the end of Section 2.2 that $V$ (the number of times the root is hit in the self-similar model on $\mathbb{T}_{3,2}$) is dominated by $Z$ (the number of times it is hit in the original model on $\mathbb{T}_{3,2}$), it follows that $\mathbb{P}(Z=\infty)=1$.  Thus we find that the frog model on $\mathbb{T}_{3,2}$ is indeed recurrent.  Hence, the proof of Theorem \ref{theorem:neat1} is complete.
\end{proof}

\section*{Acknowledgements}

The author would like to thank Toby Johnson for providing extensive background on the frog model and Konstantinos Karatapanis for several helpful conversations; thanks also to Marcus Michelen and Antonijo Mrcela for technical assistance.

\end{document}